\documentclass{article}

\usepackage[utf8]{inputenc}
\usepackage{amsmath,amssymb,amsthm}
\usepackage{enumerate}
\usepackage[a4paper,margin=3cm]{geometry}
\usepackage[all]{xy}
\usepackage{color}
\usepackage{url}
\usepackage[english]{babel}

\newcommand{\ZZ}{\mathbb Z}
\newcommand{\FF}{\mathbb F}
\newcommand{\RR}{\mathbb R}
\newcommand{\CC}{\mathbb C}
\newcommand{\calA}{\mathcal A}
\newcommand{\calC}{\mathcal C}
\newcommand{\calZ}{\mathcal Z}
\newcommand{\derC}{\mathfrak C}

\newcommand{\Fsep}{F^{\text{\rm s}}}

\newcommand{\norm}{\mathrm{N}}
\newcommand{\trace}{\mathrm{Tr}}
\newcommand{\taylisom}{\tau}
\newcommand{\taylor}{\mathrm{TS}}
\newcommand{\can}{\text{\rm can}}
\renewcommand{\sec}{\sigma}

\newcommand{\rght}{\text{\rm right}}
\newcommand{\lft}{\text{\rm left}}

\newcommand{\id}{\text{\rm id}}

\newcommand{\Tr}{\text{\rm Tr}}

\newcommand{\Frac}{\text{\rm Frac}}
\newcommand{\res}{\text{\rm res}}
\newcommand{\sres}{\text{\rm sres}}

\newcommand{\ord}{\text{\rm ord}}
\newcommand{\pp}{\mathcal P}

\newcommand{\lgcd}{\text{\sc lgcd}}
\newcommand{\rgcd}{\text{\sc rgcd}}
\newcommand{\llcm}{\text{\sc llcm}}
\newcommand{\rlcm}{\text{\sc rlcm}}

\newtheorem{theo}{Theorem}[subsection]
\newtheorem{prop}[theo]{Proposition}
\newtheorem{cor}[theo]{Corollary}
\newtheorem{lem}[theo]{Lemma}
\theoremstyle{definition}
\newtheorem{defi}[theo]{Definition}

\newtheorem{rem}[theo]{Remark}
\newtheorem{rems}[theo]{Remarks}

\begin{document}

\title{Residues of skew rational functions}
\author{Xavier Caruso}
\date\today

\maketitle

\begin{abstract}
This paper constitutes a first attempt to do analysis with skew
polynomials. Precisely, our main objective is to develop a theory
of residues for skew rational functions (which are, by definition,
the quotients of two skew polynomials). We prove in particular a
skew analogue of the residue formula and a skew analogue of the
classical formula of change of variables for residues.
\end{abstract}

\setcounter{tocdepth}{2}
\tableofcontents

\vspace{1cm}

In 1933, Ore introduced in~\cite{ore} a noncommutative variant 
of the ring of polynomials and established its first properties.
Since then, Ore's polynomials have become important mathematical
objects and have found applications in many domains of mathematics: 
abstract algebra, semi-linear algebra, linear 
differential equations (over any field), Drinfel'd modules, coding 
theory, \emph{etc.} 
Ore's polynomials have been studied by several authors: first by Ore 
himself~\cite{ore}, Jacobson~\cite{jacobson1,jacobson2} and more 
recently by Ikehata~\cite{ikehata1,ikehata2}, who proved the Ore's 
polynomial rings are Azumaya algebras in certains cases, by Lam and 
Leroy~\cite{lam,lamleroy1,lamleroy2} who defined and studied evaluation 
of Ore's polynomials, and by many others. Lectures including detailed 
discussions on Ore's polynomials also appear in the literature; for 
instance, one can cite Cohn's book~\cite{cohn} or Jacobson's 
book~\cite{jacobson-book}.

In the classical commutative case, polynomials are quite interesting 
because they exhibit at the same time algebraic and analytic aspects: 
typically, the Euclidean structure of polynomials rings has an algebraic 
flavour, while derivations and Taylor-like expansion formulas are highly 
inspired by analysis. However, as far as we know, analysis with Ore's 
polynomials has not been systematically studied yet.
This article aims at laying the first stone of this study by 
extending the theory of residues to the so-called \emph{skew 
polynomials}, which are a particular type of Ore polynomials.

Let $K$ be a field equipped with an automorphism $\theta$ of finite 
order~$r$. We consider the ring of skew polynomials $K[X;\theta]$ in 
which the multiplication is governed by the rule $X a = \theta(a) X$ 
for $a \in K$. The first striking result of this article is the 
construction of Taylor-like expansions in this framework: we show that 
any skew polynomial $f \in K[X;\theta]$ admits expansions of the form:
\begin{equation}
\label{eq:Taylorexpansion}
f(X) = \sum_{n = 0}^\infty \: a_n \cdot (X^r{-}z)^n
\end{equation}
for any given point~$z$ in a separable closure of $K$. Moreover, when 
$r$ is coprime with the characteristic of $K$, we equip $K[X;\theta]$ 
with a canonical derivation and interpret the coefficients $a_n$ 
appearing in Eq.~\eqref{eq:Taylorexpansion} as the values at $z$ of the 
successive divided derivatives of $f(X)$.
All the previous results extend without difficulty to skew rational 
functions, that are elements of the fraction field of $K[X;\theta]$;
in this generality, Taylor expansions take the form:
\begin{equation}
\label{eq:Taylorexpansion2}
f(X) = \sum_{n = v}^\infty \: a_n \cdot (X^r{-}z)^n
\qquad (v \in \ZZ).
\end{equation}
These results lead naturally to the notion of residue: by definition, 
the residue of $f(X)$ at $z$ is the coefficient $a_{-1}$ in the 
expansion~\eqref{eq:Taylorexpansion2}. Residues at infinity can 
also be defined in a similar fashion.

In the classical commutative setting, the theory of residues is very 
powerful because we have at our disposal many formulas, allowing for a 
complete toolbox for manipulating them easily and efficiently. In
this article, we shall prove that residues of skew rational functions
also exhibit interesting formulas, that are:
\begin{itemize}
\item (\emph{cf} Theorems~\ref{theo:commutativeresidue} 
and~\ref{theo:residue})
a residue formula, relating all
the residues (at all points) of a given skew rational function,
\item (\emph{cf} Theorems~\ref{theo:chvar} and~\ref{theo:chvarcan})
a formula of change of variables, expliciting how residues 
behave under an endomomorphism of $\Frac(K[X;\theta])$.
\end{itemize}
Our theory of residues has interesting applications to coding 
theory as it allows for a nice description of duals of linearised
Reed-Solomon codes (including Gabidulin codes) which have been
recently defined by Martinez-Penas~\cite{martinez1}.
This application will be discussed in a forthcoming 
article~\cite{caruso-forthcoming}.

\medskip

This article is structured as follows.
In \S\ref{sec:preliminaries}, we recall several useful algebraic properties 
of rings of skew polynomials. Special attention is paid to the study of 
endomorphisms of $K[X;\theta]$ and of its fraction fields.
In \S\ref{sec:Taylor}, we focus on Taylor-like expansions of skew polynomials 
and establish the formulas~\eqref{eq:Taylorexpansion} 
and~\eqref{eq:Taylorexpansion2}.
Finally, the theory of residues (including the residue formula
and the effect under change of variables) is presented in
\S\ref{sec:residue}.

\medskip

\noindent
\textit{Convention.}
Throughout this article, all the modules over a (not necessarily 
commutative) ring $\mathfrak A$ will always be left modules, \emph{i.e.} 
additive groups equipped with a linear action of $\mathfrak A$ on the 
left.
Similarly, a $\mathfrak A$-algebra will be a (possible noncommutative) 
ring $\mathfrak B$ equipped with a ring homomorphism $\varphi : 
\mathfrak A \to \mathfrak B$. In this situation, $\mathfrak B$ becomes 
equipped with a structure of (left) $\mathfrak A$-module defined by
$a \cdot b = \varphi(a) b$ for $a \in \mathfrak A$, $b \in \mathfrak
B$.

\section{Preliminaries}
\label{sec:preliminaries}

We consider a field $K$ equipped with an 
automorphism $\theta : K \to K$ of finite order $r$.
We let $F$ be the subfield of $K$ consisting of elements $a \in K$ with 
$\theta(a) = a$. The extension $K/F$ has degree $r$ and it is Galois 
with cyclic Galois group generated by $\theta$.

We denote by $K[X;\theta]$ the Ore algebra of \emph{skew polynomials} 
over $K$.
By definition elements of $K[X;\theta]$ are usual polynomials with
coefficients in $K$, subject to the multiplication driven by the
following rule:
\begin{equation}
\label{eq:Orerule}
\forall a \in K, \quad X \cdot a = \theta(a) X.
\end{equation}
Similarly, we define the ring $K[X^{\pm 1};\theta]$: its
elements are Laurent polynomials over $K$ in the variable 
$X$ and the multiplication on them is given by~\eqref{eq:Orerule}
and its counterpart:
\begin{equation}
\label{eq:Orerule2}
\forall a \in K, \quad X^{-1} \cdot a = \theta^{-1}(a) X^{-1}.
\end{equation}
In what follows, we will write $\calA = K[X^{\pm 1};\theta]$. 
Letting $Y = X^r$, it is easily checked that the centre of $\calA$ is 
$F[Y^{\pm 1}]$; we denote it by $\calZ$. We also set $\calC = K[Y^{\pm 1}]$; 
it is a maximal \emph{commutative} subring of $\calA$.
We shall also use the notations $\calA^+$, $\calC^+$
and $\calZ^+$ for $K[X; \theta]$, $K[Y]$ and $F[Y]$ respectively.

In this section, we first review the most important algebraic 
properties of $\calA^+$ and $\calA$, following the classical 
references~\cite{ore, lam, lamleroy1, lamleroy2, cohn, jacobson-book}.
We then study endomorphisms and derivations of $\calA^+$, $\calA$ and 
some of their quotients as they will play an important role in this 
article.

\subsection{Euclidean division and consequences}

As usual polynomials, skew polynomials are endowed with a Euclidean 
division, which is very useful for elucidating the algebraic structure 
of the rings $\calA^+$ and $\calA$. The Euclidean division
relies on the notion of degree whose definition is straightforward.

\begin{defi}
The \emph{degree} of a nonzero skew polynomial $f = \sum_i a_i X^i \in 
\calA^+$ is the largest integer $i$ for which $a_i \neq 0$.

By definition, the degree of $0 \in \calA^+$ is $-\infty$.
\end{defi}

\begin{theo}
Let $A, B \in \calA^+$ with $B \neq 0$.
\begin{enumerate}[(i)]
\item There exists $Q_\rght, R_\rght \in \calA^+$, uniquely determined,
such that $A = Q_\rght \cdot B + R_\rght$ and $\deg R_\rght < \deg B$.
\item There exists $Q_\lft, R_\lft \in \calA^+$, uniquely determined,
such that $A = B \cdot Q_\lft + R_\lft$ and $\deg R_\lft < \deg B$.
\end{enumerate}
\end{theo}

We underline that, in general, $Q_\rght \neq Q_\lft$ and $R_\rght 
\neq R_\lft$. For example, in $\CC[X,\text{conj}]$ (where $\text{conj}$
is the complex conjugacy), the right and left Euclidean
divisions of $aX$ by $X-c$ (for some $a, c \in \CC$) read:
$$aX = a\cdot (X{-}c) + ac = (X{-}c)\cdot \bar a + \bar a c.$$

\begin{rem}
Without the assumption that $\theta$ has finite order, right 
Euclidean division always exists but left Euclidean division may fail to 
exist.
\end{rem}

The mere existence of Euclidean divisions has the following classical
consequence.

\begin{cor}
The ring $\calA^+$ is left and right principal.
\end{cor}

\noindent
A further consequence is the existence of left and right gcd and lcm
on $\calA^+$. They are defined in term of ideals by:
$$\begin{array}{rcl}
\calA f + \calA g = \calA \cdot \rgcd(f,g) & ; &
\calA f \cap \calA g = \calA \cdot \llcm(f,g) \medskip \\
f \calA + g \calA = \lgcd(f,g) \cdot \calA & ; &
f \calA \cap g \calA = \rlcm(f,g) \cdot \calA
\end{array}$$
for $f, g \in \calA^+$. A noncommutative version of Euclidean 
algorithm is also available and allows for an explicit and efficient 
computation of left and right gcd and lcm.

\subsection{Fraction field}
\label{ssec:frac}

For many applications, it is often convenient to be able to
manipulate quotient of polynomials, namely rational functions, as
well-defined mathematical objects. In the skew case, defining the
field of rational functions is more subtle but can be done: using
Ore condition \cite{ore-condition} (see also \cite[\S 10]{lam-book}), 
one proves that there exists a unique field $\Frac(\calA)$ containing
$\calA$ and satisfying the following universal property: for any
noncommutative ring $\mathfrak A$ and any ring homomorphism 
$\varphi : \calA \to \mathfrak A$ such that $\varphi(x)$ is invertible
for all $x \in \calA$, $x\neq 0$, there exists a unique morphism 
$\psi : \Frac(\calA) 
\to \mathfrak A$ making the following diagram commutative:
\begin{equation}
\label{diag:fracA}
\xymatrix @C=5em {
\calA \ar[d] \ar[r]^-{\varphi} & \mathfrak A \\
\Frac(\calA) \ar[ru]_-{\psi} }
\end{equation}
Under our assumption that $\theta$ has finite order the construction
of $\Frac(\calA)$ can be simplified. Indeed, we have the following
theorem.

\begin{theo}
\label{theo:fracA}
The ring $\Frac(\calZ) \otimes_\calZ \calA \simeq
\Frac(\calZ) \otimes_{\calZ^+} \calA^+$ containing $\calA$
and it satisfies the above universal property, \emph{i.e.}:
$$\Frac(\calA) = \Frac(\calZ) \otimes_\calZ \calA
= \Frac(\calZ) \otimes_{\calZ^+} \calA^+.$$
\end{theo}

For the proof, we will need the following lemma.

\begin{lem}
\label{lem:centralmultiple}
Any skew polynomial $f \in \calA$ has a left multiple and a right
multiple in $\calZ$.
\end{lem}

\begin{proof}
If $f = 0$, the lemma is obvious. Otherwise, the quotient
$\calA / f\calA$ is a finite dimension vector space over $F$. Hence,
there exists a nontrivial relation of linear dependence of the form:
$$a_0 + a_1 Y + a_2 Y^2 + \cdots + a_n Y^n \in f \calA
\qquad (a_i \in F).$$
In other words, there exists $g \in \calA$ such that $fg = N$
with $N = a_0 + \cdots + a_n Y^n$. In particular $fg \in \calZ$,
showing that $f$ has a right multiple in $\calZ$.
Multiplying the relation $fg = N$ by $g$ on the left, we get
$gfg = Ng = gN$. Simplifying now by $g$ on the left, we are left
with $gf = N$, showing that $f$ has a left mutiple in $\calZ$ as
well.
\end{proof}

\begin{proof}[Proof of Theorem~\ref{theo:fracA}]
Clearly $\Frac(\calZ) \otimes_\calZ \calA$ contains $\calA$. Let
us prove now that it is a field.
Reducing to the same denominator, we remark that any 
element of $\Frac(\calZ) \otimes_\calZ \calA$ can be written as
$D^{-1} \otimes f$ with $D \in \calZ$ and $f \in \calA$. We
assume that $f \neq 0$.
By Lemma~\ref{lem:centralmultiple}, there exists $g \in \calA$
such that $fg \in \calZ$. Letting $N = fg$, one checks that
$N^{-1} \otimes gD$ is a multiplicative inverse of $D^{-1} \otimes f$.

Consider now a noncommutative ring $\mathfrak A$ together with a
ring homomorphism $\varphi : \calA \to \mathfrak A$ such that
$\varphi(x)$ is invertible for all $x \in \calA$, $x \neq 0$. If 
$\psi :
\Frac(\calZ) \otimes_\calZ \calA \to \mathfrak A$ is an extension
of $\varphi$, it must satisfy:
\begin{equation}
\label{eq:psifracA}
\psi\big(D^{-1} \otimes f\big) = \varphi(D)^{-1} \cdot \varphi(f).
\end{equation}
This proves that, if such an extension exists, it is unique. On the
other hand, using that $\calZ$ is central in $\calA$, one checks that
the formula \eqref{eq:psifracA} determines a well-defined ring 
homomorphism $\Frac(\calA) \to \mathfrak A$ making the diagram
\eqref{diag:fracA} commutative.
\end{proof}

The notion of degree extends without difficulty to skew rational
functions: if $f = \frac g D \in \Frac(\calA)$ with $g \in \calA^+$
and $D \in \calZ^+$, we define $\deg f = \deg g - \deg D$.
This definition is not ambiguous because an equality of the form
$\frac g D = \frac {g'}{D'}$ implies $g D' = g' D$ (since
$D$ and $D'$ are central) and then
$\deg g + \deg D' = \deg g' + \deg D$, that is
$\deg g - \deg D = \deg g' - \deg D'$.

\subsection{Endomorphisms of Ore polynomials rings}
\label{ssec:endoOre}

The aim of this subsection is to classify and derive interesting
structural properties of the endomorphisms of various rings of skew
polynomials.

\subsubsection{Classification}

Given an integer $n \in \ZZ$ and a Laurent polynomial $C \in \calC$ 
written as $C = \sum_i a_i X^i$, we define $\theta(C) = \sum_i 
\theta(a_i) X^i$. The morphism $\theta$ extends to 
$\Frac(\calC$). For $n \geq 0$ and $C \in \Frac(\calC)$, we set:
$$\norm_n(C) 
= C \cdot \theta(C) \cdots \theta^{n-1}(C)$$
and, when $C \neq 0$, we extend the definition of $\norm_n$ to
negative $n$ by:
$$\norm_n(C) 
 = \theta^{-1}(C^{-1}) \cdot \theta^{-2}(C^{-1}) \cdots \theta^n(C^{-1})$$
We observe that $\norm_0(C) = 1$ and $\norm_1(C) = C$ for all $C \in 
\calC$. Moreover, when $n = r$, the mapping $\norm_r$ is the norm from 
$\Frac(\calC)$ to $\Frac(\calZ)$. In particular $\norm_r(C) \in 
\Frac(\calZ)$ for all $C \in \Frac(\calC$).

\begin{theo}
\label{theo:endoOre}
Let $\gamma : \calA^+ \to \calA^+$ (resp. $\gamma : \calA \to \calA$, 
resp. $\gamma : \Frac(\calA) \to \Frac(\calA)$) be a morphism of 
$K$-algebras. Then there exists a uniquely determined element $C \in 
\calC^+$ (resp. invertible\footnote{We notice that the invertible 
elements of $\calC$ are exactly those of the form $a Y^n$ with $a \in 
K$, $a \neq 0$ and $n \in \ZZ$.} element $C \in \calC$, resp. nonzero 
element $C \in \Frac(\calC)$) such that
\begin{equation}
\label{eq:gamma}
\gamma\Big(\sum_i a_i X^i\Big) = \sum_i a_i (CX)^i = 
\sum_i a_i \norm_i(C) X^i.
\end{equation}
Conversely any element of $C$ as above gives rise to a well-defined
endomorphism of $\calA^+$ (resp. $\calA$, resp. $\Frac(\calA)$).
\end{theo}

\begin{rem}
\label{rem:endoFracA}
An endomorphism of $\Frac(\calA)$ is entirely determined by 
Eq.~\eqref{eq:gamma}. Indeed, by definition, the datum of 
$\gamma : \Frac(\calA) \to \Frac(\calA)$ is equivalent to 
the datum of a morphism $\tilde \gamma : \calA \to \Frac(\calA)$
with the property that $\tilde\gamma(f) \neq 0$ whenever $f \neq 0$.
Moreover, in the above equivalence, $\tilde \gamma$ appears as the
restriction of $\gamma$ to $\calA$. This shows, in particular, 
that $\gamma$ is determined by its restriction to $\calA$.
\end{rem}

\begin{proof}[Proof of Theorem~\ref{theo:endoOre}]
Unicity is obvious since $C$ can be recovered thanks to the formula
$C = \gamma(X) X^{-1}$.

We first consider the case of an endomorphism of $\calA^+$.
Write $\gamma(X) = \sum_i c_i X^i$ with $c_i \in K$. Applying $\gamma$ 
to the relation \eqref{eq:Orerule}, we obtain:
$$\sum_i c_i \theta^i(a) \cdot X^{i+1} 
= \sum_i c_i \theta(a) \cdot X^{i+1}$$
for all $a \in K$.
Identifying the coefficients, we end up with $c_i \theta^i(a) =
c_i \theta(a)$. Since this equality must hold for all $a$, we find that
$c_i$ must vanish as soon as $i \not\equiv 1 \pmod r$. Therefore,
$\gamma(X) = CX$ for some element $C \in \calC^+$. An easy induction
on $i$ then shows that $\gamma(X^i) = \norm_i(C) X^i$ for all $i$,
implying eventually \eqref{eq:gamma}. Conversely, it is easy to 
check that Eq.~\eqref{eq:gamma} defines a morphism of $K$-algebras.

For endomorphisms of $\calA$, the proof is exactly the same, except
that we have to justify further that $C$ is invertible. This comes
from the fact that $X\:\gamma(X^{-1})$ has to be an inverse of $C$.

We now come to the case of endomorphisms of $\Frac(\calA)$.
Writing $\gamma(X) = f D^{-1}$ with $f \in \calA^+$ and $D \in 
\calZ^+$ and repeating the proof above, we find that $f X^{-1} \in 
\calC$. Thus $\gamma(X) = C X$ with $C \in \Frac(\calC)$. As before,
$C$ cannot vanish because it admits $X \: \gamma(X^{-1})$ as an 
inverse. 
From the fact that $\gamma$ is an endomorphism of $K$-algebras,
we deduce that $\gamma_{|\calA}$ is given by Eq.~\eqref{eq:gamma}. 
Conversely, we need to justify that the morphism $\gamma$
defined by Eq.~\eqref{eq:gamma} extends to $\Frac(\calA)$. After
Remark~\ref{rem:endoFracA}, it is enough to check that $\gamma(f) 
\neq 0$ when $f \neq 0$, which can be seen by comparing degrees.
\end{proof}

For $C \in \Frac(\calC)$, $C \neq 0$, we let $\gamma_C : \Frac(\calA) \to 
\Frac(\calA)$ denote the endomorphism of Theorem~\ref{theo:endoOre} ($X 
\mapsto CX$). When $C$ lies in $\calC^+$ (resp. when $C$ is invertible
in $\calC$), $\gamma_C$ stabilized $\calA^+$ (resp. $\calA$); when this
occurs, we will continue to write $\gamma_C$ for the endomorphism 
induced on $\calA^+$ (resp. on $\calA$).
We observe that $\gamma_C$ takes $Y$ to:
$$\norm_r(C)\cdot Y = 
\norm_{\Frac(\calC)/\Frac(\calZ)}(C) \cdot Y \in \Frac(\calZ)$$ 
and, therefore, maps 
$\Frac(\calZ)$ to itself. In other words, any endomorphism of $K$-algebras
of $\Frac(\calA)$ stabilizes the centre. This property holds similarly 
for endomorphism of $\calA^+$ and endomorphisms of $\calA$.

\begin{prop}
\label{prop:endoCalg}
For $C \in \Frac(\calC)$, the following assertions are equivalent:
\begin{enumerate}[(i)]
\item \label{item:Calg} $\gamma_C$ is a morphism of $\calC$-algebras,
\item \label{item:norm1} $\norm_{\Frac(\calC)/\Frac(\calZ)}(C) = 1$,
\item \label{item:conj} there exists $U \in \Frac(\calC)$, $U \neq 0$ 
such that $\gamma_C(f) = U^{-1} f U$ for all $f \in \Frac(\calA)$.
\end{enumerate}
\end{prop}

\begin{proof}
If $\gamma_C$ is an endomorphism of $\calC$-algebras, it must act
trivially on $\calZ$, implying then \eqref{item:norm1}.
By Hilbert's Theorem 90, if $C \in \Frac(\calC)$ has norm $1$, it can be written
as $\frac{\theta(U)} u$ for some $U \in \Frac(\calC)$, $U \neq 0$; 
\eqref{item:conj}
follows. Finally it is routine to check that \eqref{item:conj} implies
\eqref{item:Calg}.
\end{proof}

For endomorphisms of $\calA^+$ and $\calA$, 
Proposition~\ref{prop:endoCalg} can be made more precise.

\begin{prop}
\label{prop:endoCalg2}
For $C \in \calC$, the following assertions are equivalent:
\begin{enumerate}[(i)]
\item $\gamma_C$ is a morphism of $\calC$-algebras,
\item $\norm_{\calC/\calZ}(C) = 1$,
\item[(ii')] \label{item:constant} $C \in K$ and $\norm_{K/F}(C) = 1$,
\item there exists $u \in K$, $u \neq 0$
such that $\gamma_C(f) = u^{-1} f u$ for all $f \in \Frac(\calA)$.
\end{enumerate}
\end{prop}

\begin{proof}
The proof is the same as that of Proposition~\ref{prop:endoCalg},
except that we need to justify in addition that any element $C \in \calC$
of norm $1$ needs to be a constant. This follows by comparing degrees.
\end{proof}

\begin{cor}
Any endomorphism of $\calC$-algebras of $\calA^+$ (resp. $\calA$, 
resp. $\Frac(\calA)$) is an isomorphism.
\end{cor}

\begin{proof}
The case of $\calA^+$ (resp. $\calA$) follows directly from
Proposition~\ref{prop:endoCalg2}. For $\Frac(\calA)$, we check
that if $\gamma_C$ is an endomorphism of $\calC$-algebra then 
$\gamma_{C^{-1}}$ is also (it is a consequence of 
Proposition~\ref{prop:endoCalg}) and
$\gamma_C \circ \gamma_{C^{-1}} = \gamma_{C^{-1}} \circ \gamma_C 
= \id$.
\end{proof}

\subsubsection{Morphisms between quotients}

Let $N \in \calZ^+$ be a nonconstant polynomial with a nonzero 
constant term. The principal ideals generated by $N$ in $\calA^+$ 
and $\calA$ respectively are two-sided, so that the quotients 
$\calA^+/N\calA^+$ and $\calA/N\calA$ inherit a structure 
of $K$-algebra. By our assumptions on $N$, they are moreover isomorphic.
We consider in addition a
commutative algebra $\calZ'$ over~$\calZ$. We let $\theta$ act
on $\calZ^+ \otimes_\calZ \calC$ by $\id \otimes \theta$ and we
extend the definition of $\gamma_C$ to all elements $C \in \calZ' 
\otimes_{\calZ} \calC$. Namely, for $C$ as above, we define
$\gamma_C : \calA^+ \to \calZ^+ \otimes_\calZ \calA$ by
$$\gamma_C\Big(\sum_i a_i X^i\Big) = \sum_i a_i (CX)^i = 
\sum_i a_i \norm_i(C) X^i.$$

\begin{theo}
\label{theo:endoquotOre}
Let $N_1, N_2 \in \calZ^+$ be two nonconstant polynomials with 
nonzero constant terms.
Let $\gamma : \calA/N_1 \calA \to \calZ' \otimes_{\calZ} \calA/N_2\calA$ 
be a morphism of
$K$-algebras. Then $\gamma = \gamma_C \mod {N_2}$ for some
element $C \in \calZ' \otimes_\calZ \calC$ with the
property that $N_2$ divides $\gamma_C(N_1)$. Such an element
$C$ is uniquely determined modulo $N_2$.

Moreover, the following assertions are equivalent:
\begin{enumerate}[(i)]
\item $\gamma$ is a morphism of $\calC$-algebras,
\item $\norm_{\calZ' \otimes_\calZ \calC/\calZ'}(C) \equiv 1 
\pmod{N_2}$.
\item there exists $U \in \calZ' \otimes_\calZ \calC/N_2\calC$, 
$U$ invertible
such that $\gamma(f) = U^{-1} f U$ for all $f \in \calA/N_1\calA$.
\end{enumerate}
\end{theo}

\begin{proof}
The proof is entirely similar to that of Theorem~\ref{theo:endoOre} and 
Proposition~\ref{prop:endoCalg}. Note that, for the point (iii), 
Hilbert's Theorem 90 applies because the extension $\calZ' \otimes_\calZ
\calC/N_2\calC$ of $\calZ'/N_2\calZ'$ is a cyclic Galois covering.
\end{proof}

As an example, let us have a look at the case where $\calZ' = \calZ$
and $N_1$ and $N_2$
have $Y$-degree $1$. Write $N_1 = Y - z_1$ and $N_2 = Y - z_2$ with
$z_1 \neq 0$ and $z_2 \neq 0$. By Theorem~\ref{theo:endoquotOre}, any 
morphism $\gamma : \calA/N_1 \calA \to \calA/N_2\calA$ has the form
$X \mapsto cX$ for an element $c \in K$ with the property that:
\begin{equation}
\label{eq:endodeg1}
z_1 = \norm_{K/F}(c) \cdot z_2.
\end{equation}
Obviously, Eq.~\eqref{eq:endodeg1} implies that $c$ does not vanish.
Hence, any morphism $\gamma$ as above is automatically an isomorphism.
Moreover, Eq.~\eqref{eq:endodeg1} again shows that $\frac{z_1}{z_2}$ 
must be a norm in the extension $K/F$.
Conversely, if $\frac{z_1}{z_2}$ is the norm of an element $c \in K$,
the morphism $\gamma_C$ induces an isomorphism between 
$\calA/N_1 \calA$ to $\calA/N_2\calA$. We have then proved the
following proposition.

\begin{prop}
\label{prop:endodeg1}
Let $z_1$ and $z_2$ be two nonzero elements of $F$.
There exists a morphism $\calA/(Y{-}z_1) \calA \to \calA/(Y{-}z_2)\calA$
if and only if $\frac{z_1}{z_2}$ is a norm in the extension $K/F$.
Moreover, when this occurs, any such morphism is an isomorphism.
\end{prop}

\subsubsection{The section operators}

For $j \in \ZZ$, we define the \emph{section operator} $\sec_j : \calA
\to \calC$ by the formula:
$$\sec_j\Big(\sum a_i X^i\Big) = \sum_i a_{j+ir} Y^i.$$
For $0 \leq j < r$ and $f \in \calA$, we notice that
$\sec_j(f)$ is the $j$-th coordinate of $f$ in the canonical basis 
$(1, X, X^2, \ldots, X^{r-1})$ of 
$\calA$ over $\calC$. When $j \geq 0$, we observe that $\sec_j$ takes 
$\calA^+$ to $\calC^+$ and then induces a mapping 
$\calA^+ \to \calC^+$ that, in a slight abuse of 
notations, we will continue to call $\sec_j$.

\begin{lem}
\label{lem:form:sec}
For $f \in \calA$, $C \in \calC$ and $j \in \ZZ$, 
the following identities hold:
\begin{enumerate}[(i)]
\item $f = \sum_{j=0}^{p-1} \sec_j(f) X^j$,
\item $\sec_j(fC) = \sec_j(f){\cdot}\theta^j(C)$ and 
$\sec_j(fX) = \sec_{j-1}(f)$,
\item $\sec_j(Cf) = C{\cdot}\sec_j(f)$ and 
$\sec_j(Xf) = \theta(\sec_{j-1}(f))$,
\item $\sec_{j-r}(f) = Y {\cdot} \sec_j(f)$.
\end{enumerate}
\end{lem}

\begin{proof}
It is an easy checking.
\end{proof}

Lemma~\ref{lem:form:sec} ensures in particular that $\sec_0$ is 
$\calC$-linear and the $\sec_j$'s are $\calZ$-linear for all $j \in 
\ZZ$. Consequently, for any integer $j$, the operator $\sec_j$ induces a 
$\Frac(\calC)$-linear mapping $\Frac(\calA) \to \Frac(\calC)$. 
Similarly, for any $N \in \calZ$ and any integer $j$, it also induces a 
$(\calZ/N\calZ)$-linear mapping $\calA/N\calA \to \calC/N\calC$. 
Tensoring by a commutative $\calZ$-algebra $\calZ'$, we find that
$\sec_j$ induces also a $(\calZ'/N\calZ')$-linear mapping $\calZ' 
\otimes_\calZ \calA/N\calA \to \calZ' \otimes_\calZ \calC/N\calC$. In a 
slight abuse of notations, we will continue to denote by $\sec_j$ all 
the extensions of $\sec_j$ defined above.

It worths remarking that the section operators satisfy special 
commutation relations with the morphisms $\gamma_C$, namely:

\begin{lem}
\label{lem:endosec}
For $C \in \Frac(\calC)$ (resp. $C \in \calZ' \otimes_\calZ \calC$)
and $j \in \ZZ$, we have the relation
$\sec_j \circ \gamma_C = \norm_j(C) \cdot (\gamma_C \circ \sec_j)$.
\end{lem}

\begin{proof}
Let $f \in \calA^+$ and write
$f = \sum_{i=0}^{r-1} \sec_i(f) X^i$. Applying $\gamma_C$ to this
relation, we obtain:
$$\gamma_C(f) = \sum_{i=0}^{r-1} \gamma_C \circ \sec_i(f) \cdot 
\norm_j(X) \: X^i.$$
Applying now $\sec_j$, we end up with
$\sec_j \circ \gamma_C(f) = \gamma_C \circ \sec_j(f) \cdot \norm_j(X)$.
This proves the lemma.
\end{proof}

\noindent
Using Lemma \ref{lem:endosec}, it is possible to construct some
quantities that are invariant under all $\gamma_C$, that is, after
Theorem~\ref{theo:endoOre} or \ref{theo:endoquotOre}, under all 
morphisms of $K$-algebras. 
Precisely, for a tuple of integers $(j_1, \ldots, j_m) \in \ZZ^m$, 
we define:
$$\sec_{j_1, \ldots, j_m} = \sec_{j_1} 
{}\cdot{} (\theta^{j_1} \circ \sec_{j_2})
{}\cdot{} (\theta^{j_1 + j_2} \circ \sec_{j_3})
\:{}\cdots{}\: (\theta^{j_1 + \cdots + j_{m-1}} \circ \sec_{j_m}) 
\,\,:\,\, \Frac(\calA) \to \Frac(\calC).$$

\begin{prop}
\label{prop:endoinv}
Let $\gamma : \calA^+ \to \calA^+$ (resp. $\gamma : \calA \to \calA$, resp. 
$\gamma : \Frac(\calA) \to \Frac(\calA)$, resp. $\gamma : \calA/N_1 
\calA \to \calZ' \otimes_\calZ \calA/N_2\calA$ with $\calZ'$, $N_1, N_2$ 
as in Theorem~\ref{theo:endoquotOre}).
Let $(j_1, \ldots, j_m) \in \ZZ^m$.

\begin{enumerate}[(i)]
\item If $\gamma$ is a morphism of $K$-algebras, 
then $\gamma$ commutes with $\sec_{j_1, \ldots, j_m}$ as soon as $j_1 + \cdots
+ j_m = 0$.
\item If $\gamma$ is a morphism of $\calC$-algebras, 
then $\gamma$ commutes with $\sec_{j_1, \ldots, j_m}$ as soon as $j_1 + \cdots
+ j_m \equiv 0 \pmod r$.
\end{enumerate}
\end{prop}

\begin{proof}
By Theorem~\ref{theo:endoOre} or \ref{theo:endoquotOre}, it is enough
to prove the Proposition when $\gamma = \gamma_C$ for some $C$.
By Lemma~\ref{lem:endosec}, combined with the relation
$\norm_{j+j'}(C) = \norm_j(C) \cdot \theta^j\big(\norm_{j'}(C)\big)$
(for $j,j' \in \ZZ$), we find:
$$\sec_{j_1, \ldots, j_m} \circ \gamma_C = 
\norm_{j_1 + \cdots + j_m}(C) \cdot 
\big(\gamma_C \circ \sec_{j_1, \ldots, j_m}\big).$$
The first assertion follows while the second is a direct consequence
of the caracterisation of morphisms of $\calC$-algebras given by
Proposition~\ref{prop:endoCalg2} or Theorem~\ref{theo:endoquotOre}.
\end{proof}

\subsection{Derivations over Ore polynomials rings}
\label{ssec:derivOre}

Given a (possibly noncommutative) ring $\mathfrak A$ and a $\mathfrak 
A$-algebra $\mathfrak B$, we recall that a \emph{derivation} $\partial : 
\mathfrak A \to \mathfrak B$ is an additive mapping satisfying the 
Leibniz rule:
$$\partial(xy) = \partial(x) y + x \partial(y)
\quad (x,y \in \mathfrak A).$$
One checks that the subset $\mathfrak C \subset \mathfrak A$ consisting 
of elements $x$ with $\partial(x) = 0$ is actually a subring de $\mathfrak 
A$. It is called the \emph{ring of constants}. A derivation $\partial : 
\mathfrak A \to \mathfrak B$ with ring of constants $\mathfrak C$ is 
$\mathfrak C$-linear.

\subsubsection{Classification}

As we classified endomorphisms of $K$-algebras in \S\ref{ssec:endoOre}, it is 
possible to classify $K$-linear derivations over Ore rings. For $C \in
\Frac(\calC)$, and $n \in \ZZ$, we define:
$$\begin{array}{r@{\hspace{0.5ex}}ll}
\trace_n(C) 
 & = C + \theta(C) + \cdots + \theta^{n-1}(C) & \text{if } n \geq 0 \smallskip \\
 & = -\theta^{-1}(C) - \theta^{-2}(C) - \cdots - \theta^n(C) & \text{if } n < 0
\end{array}$$
We observe that $\trace_r$ is the trace from $\Frac(\calC)$ to $\Frac(\calZ)$. 
In particular, it takes its values in $\Frac(\calZ)$.

\begin{prop}
\label{prop:derivOre}
Let $\partial : \calA^+ \to \calA^+$ (resp. $\partial :
\calA \to \calA$, resp. $\partial : \Frac(\calA) \to \Frac(\calA)$) be 
a $K$-linear derivation, \emph{i.e.} a derivation whose ring of constants 
contains $K$.
Then, there exists a uniquely determined $C \in \calC^+$ (resp. $C \in 
\calC$, resp. $C \in \Frac(\calC)$) such that:
\begin{equation}
\label{eq:partial}
\partial\Big(\sum_i a_i X^i\Big) = \sum_i a_i \trace_i(C) X^i.
\end{equation}
Conversely, any such $C$ gives rise to a unique derivation of $\calA^+$
(resp. $\calA$, resp. $\Frac(\calA)$).
\end{prop}

\begin{proof}
Unicity is clear since $C = \partial(X) X^{-1}$.

Let $\partial$ be a $K$-linear derivation as in the proposition. Applying
$\partial$ to the equality $Xa = \theta(a) X$ ($a \in K$), we get
$\partial(X) {\cdot} a = \theta(a) {\cdot} \partial(X)$. Writing 
$\partial(X) = \sum_i c_i X^i$, we deduce $c_i \theta^i(a) = c_i 
\theta(a)$ for all index $i$,
showing that $c_i$ has to vanish when $i \not\equiv 1 \pmod r$. 
Thus $\partial(X) = CX$ for some $C \in \calC^+$ (resp. $C \in
\calC$). A direct computation then shows that:
$$\partial(X^2) = X {\cdot}\partial(X) + \partial(X) {\cdot}X = 
X C X + C X^2  = \big(C + \theta(C)\big) X^2 = \trace_2(C) X^2$$
and, more generally, an easy induction leads to $\partial(X^i) = 
\trace_i(C) 
X^i$ for all $i \geq 0$.
In the cases of $\calA$ and $\Frac(\calA)$, we can also compute
$\partial(X^i)$ when $i$
is negative. For this, we write:
$$0 = \partial(1) = \partial(X^{-1} X) = \partial(X^{-1}) X + X^{-1} C X$$
from what we deduce that $\partial(X^{-1}) = - X^{-1} C = -\theta^{-1}(C) 
X^{-1} = \trace_{-1}(C) X^{-1}$. As before, an easy induction on $i$ then 
gives $\partial(X^i) = \trace_i(C)\:X^i$ for all negative $i$. We deduce
that Eq.~\eqref{eq:partial} holds.

For the converse, we first check that Eq.~\eqref{eq:partial} defines a 
derivation on $\calA$. In the
case of $\Frac(\calA)$, we need to justify in addition that $\partial$ 
(given by Eq.~\eqref{eq:partial}) extends uniquely to $\Frac(\calA)$. This 
is a consequence of the following formula:
$$\partial\left(\frac f D\right) = 
\frac{\partial(f) \:D \,+\, f \:\partial(D)}{D^2}
\qquad (f\in \calA, \, D \in \calZ)$$
which holds true because $D$ is central.
\end{proof}

Let $\partial_C : \Frac(\calA) \to \Frac(\calA)$ denote the derivation 
of Proposition \ref{prop:derivOre}. We have:
$$\partial_C(Y) = \trace_r(C) 
\cdot Y = \trace_{\Frac(\calC)/\Frac(\calZ)}(C) \cdot Y \in 
\Frac(\calZ).$$
We deduce that $\partial_C$ stablizes $\Frac(\calC)$ and 
$\Frac(\calZ)$ and acts on these rings as the derivation 
$\trace_{\Frac(\calC)/\Frac(\calZ)}(C) \cdot Y \cdot \frac d{dY}$.

\begin{prop}
\label{prop:derivClin}
For $C \in \Frac(\calC)$, the following assertions are equivalent:
\begin{enumerate}[(i)]
\item $\partial_C$ is $\calC$-linear,
\item $\trace_{\Frac(\calC)/\Frac(\calZ)}(C) = 0$,
\item there exists $U \in \Frac(\calC)$ such that $\partial_C(f) = fU - Uf$
for all $f \in \Frac(\calA)$.
\end{enumerate}
\end{prop}

\begin{proof}
The equivalence between (i) and (ii) is clear by what we have seen 
before. If (ii) holds, then the additive version of Hilbert's Theorem 90 
ensures that $C$ can be written as $\theta(U) - U$ with $U \in \Frac(\calC)$.
Then $\partial_C(X^i) = \trace_i(\theta(U) - U) X^i = \theta^i(U) X^i
- U X^i = X^i U - U X^i$ for all integer $i$. By $K$-linearity, we
deduce that $\partial_C(f)  = fU - Uf$ for all $f \in \calA$, 
implying (iii). Finally, if (iii) holds, $\partial_C$ clearly 
vanishes on $\calC$, implying (i).
\end{proof}

\subsubsection{Extensions of the canonical derivation $\frac d{dY}$}

An important case of interest occurs when 
$\trace_{\Frac(\calC)/\Frac(\calZ)}(C) = Y^{-1}$, as $\partial_C$ then induces the 
standard derivation $\frac d {dY}$ on $\Frac(\calC)$. When $p$ does not divide 
$r$, a distinguished element $C$ satisfying this condition is $C = 
r^{-1} Y^{-1}$.

\begin{defi}
\label{defi:dercan}
When $p$ does not divide $r$, we set 
$\partial_{Y,\can} = \partial_{r^{-1} Y^{-1}}$. Explicitely:
$$\partial_{Y,\can}\Big(\sum_i a_i X^i\Big) =
r^{-1} \cdot \sum_i i a_i X^{i-r}.$$
\end{defi}

An interesting feature of the derivation $\partial_{Y,\can}$ is that 
its $p$-th power vanishes (as we can check easily by hand).
This property will be very pleasant for us in \S \ref{sec:Taylor} 
when we will define Taylor expansions of skew polynomials.
Unfortunately, it seems that there is no simple analogue of $\partial_{Y,\can}$
when $p$ divides $r$, as shown by the following proposition.

\begin{prop}
\label{prop:pthpowerzero}
Let $C \in \Frac(\calC)$ with
$\trace_{\Frac(\calC)/\Frac(\calZ)}(C) = Y^{-1}$ and $\partial_C^p = 0$.
Then $p$ does not divide $r$.
\end{prop}

\begin{proof}
Our assumptions ensure that $\partial_C$ induces the derivation
$\frac d{dY}$ on $\Frac(\calC)$.
For $i \in \{1, 2, \ldots, p\}$, we define $C_i = \partial_C^i(X) \: 
X^{-1}$. A direct computation shows that:
\begin{equation}
\label{eq:recCi}
C_1 = C \quad ; \quad C_{i+1} = \frac{d C_i}{dY} + C_i \: C.
\end{equation}
In particular, we deduce that $C_i \in \Frac(\calC)$ for all $i$. 
We claim that $C$ has at most a simple pole at $0$. Indeed, if we
assume
by contradiction that $C$ has a pole of order $v \geq 2$ at $0$, we
would deduce that $C_i$ has a pole of order $vi$ at 
$0$ for $i \in \{1, \ldots, p\}$, contradicting the fact that $C_p$
vanishes.
We can then write $C = a Y^{-1} + O(1)$ with $a \in K$. The recurrence
relation \eqref{eq:recCi} shows that, for $i \in \{1, \ldots, p\}$,
we have $C_i = a_i Y^{-i} + O(Y^{-i+1})$ where the coefficients $a_i$'s
satisfy:
$$a_1 = a \quad ; \quad 
a_{i+1} = -i a_i + a_i a = a_i \cdot (a - i)$$
Hence $a_p = a \cdot (a - 1) \cdots (a - (p{-}1)) = a^p - a$. In order 
to guarantee that $a_p$ vanishes, we then need $a \in \FF_p \subset F$.
Taking the trace, we obtain $\trace_{\Frac(\calC)/\Frac(\calZ)}(C) = r a \: Y^{-1} + 
O(1)$. Thus $ra = 1$ in $F$ and $p$ cannot divide $r$.
\end{proof}

\begin{rem}
With the notation of the proof above, $C_p$ is the function by which 
the $p$-curvature of the linear differential equation $y' = Cy$ acts.
With this reinterpretation, one can use Jacobson identity (see 
Lemma~1.4.2 of \cite{vdp}) to get a closed formula for $C_p$, which 
reads:
$$C_p = \frac{d^{p-1}C}{dY^{p-1}} + C^p.$$
\end{rem}

\subsubsection{Derivations over quotients of Ore rings}

Following \S \ref{ssec:endoOre}, we propose to classify $K$-linear 
derivations $\calA/N_1 \calA \to \calA/N_2 \calA$. However, we need to 
pay attention in this case that such derivations are only defined when 
$\calA/N_2 \calA$ is an algebra over $\calA/N_1 \calA$, that is when 
$N_1$ divides $N_2$.
As in \S \ref{ssec:endoOre}, we consider in addition a commutative 
$\calZ$-algebra~$\calZ'$. We extend readily the definition of 
$\partial_C$ to an element $C \in \calZ' \otimes_\calZ \Frac(\calA)$.

\begin{prop}
Let $N_1, N_2 \in \calZ^+$ be two nonconstant polynomials with 
nonzero constant terms. We assume that $N_1$ divides $N_2$.
Let $\calZ'$ be a commutative $\calZ$-algebra.

Let $\partial : \calA/N_1 \calA \to \calZ' \otimes_\calZ 
\calA/N_2\calA$ be $K$-linear
derivation. Then $\partial = \partial_C \mod {N_2}$ for some
element $C \in \calZ' \otimes_\calZ \calC$ with the
property that $N_2$ divides $\partial_C(N_1)$.
Such an element $C$ is uniquely determined modulo $N_2$.

Moreover, the following assertions are equivalent:
\begin{enumerate}[(i)]
\item $\partial$ is a $\calC$-linear
\item $\trace_{\calZ' \otimes_\calZ \calC/\calZ'}(C) \equiv 0 \pmod{N_2}$.
\item there exists $U \in \calZ' \otimes_\calZ \calC/N_2\calC$ such that 
$\partial(f) = f U - U f$ for all $f \in \calA/N_1\calA$.
\end{enumerate}
\end{prop}

\begin{proof}
It is entirely similar to the proofs of Propositions~\ref{prop:derivOre}
and~\ref{prop:derivClin}.
\end{proof}

\section{Taylor expansions}
\label{sec:Taylor}

The aim of this subsection is to show that skew polynomials admit
Taylor expansion around any closed point of $F$ and to study its
properties. Besides, when $r$ is coprime to $p$, we will prove 
that the Taylor expansion is canonical and given by a Taylor-like
series involving the successive divided powers of the derivation
$\partial_\can$.

\subsection{The commutative case: reminders}
\label{sssec:taylorreminders}

By definition, we recall that the \emph{Taylor expansion} of a 
Laurent polynomial $f \in \calC$ around a point $c \in K$, $c \neq 0$
is the series:
\begin{equation}
\label{eq:taylorclassical}
\sum_{n=0}^\infty f^{[n]}(c) \: T^n
\end{equation}
where $T$ is a formal variable playing the role of $Y{+}c$ and the
notation $f^{[n]}$ stands for the $n$-th divided derivative of $f$
defined by:
$$\Big(\sum_i a_i Y^i\Big)^{[n]} = \sum_i \binom i n \cdot 
a_i \: Y^{i-n}
\qquad (a_i \in K).$$
We recall also that the $n$-th divided derivative satisfies the 
following Leibniz-type relation:
$$(fg)^{[n]} = \sum_{m=0}^n f^{[m]} g^{[n-m]}
\qquad (f,g \in \calC^+)$$
from what we deduce that the mapping $\calC \to K[[T]]$ 
taking a Laurent polynomial to its Taylor expansion is a homomorphism of 
$K$-algebras. Even better, it induces an isomorphism:
$$\taylisom_c^{\cal C} : \,\,
\varprojlim_{m > 0} \calC/(Y{-}c)^m \: \calC \,\, \simeq \,\, K[[T]].$$

More generally, let us consider 
an irreducible separable polynomial
$N \in \calC$. Let also $c \in \calC/N\calC$ be the image of $X$,
which is a root of $N$ by construction.
In this generality, the Taylor expansion around $c$ is well-defined and
induces a homomorphism of $K$-algebras $\calC \to (\calC/N\calC)[[T]]$,
inducing itself an isomorphism:
$$\taylisom_c^{\cal C} : \,\,
\varprojlim_{m > 0} \calC/N^m\calC \,\, 
\simeq \,\, (\calC/N\calC)[[T]].$$
The image of $N$ under this isomorphism is a series of valuation $1$.
As a consequence, twisting by an automorphism of $(\calC/N\calC)[[T]]$,
there exists an isomorphism of $K$-algebras:
$$\taylisom_N^{\cal C} : \,\,
\varprojlim_{m > 0} \calC/N^m\calC \,\, 
\simeq \,\, (\calC/N\calC)[[T]]$$
mapping $N$ to $T$ and inducing the identity map $\calC/N\calC 
\to \calC/N\calC$ after reduction modulo $N$ on the left and modulo $T$
on the right. Moreover $\taylisom_N^{\cal C}$ is uniquely determined by
these properties. In addition, we observe that when $N = Y{-}c$ is a 
polynomial of degree $1$, the isomorphisms $\taylisom_{Y-c}^{\cal C}$ 
and $\taylisom_c^\calC$ agree.

It turns out that the existence of the unicity of $\taylisom_N^{\calC}$ 
continues to hold under the sole assumption that $N$ is separable; this 
can be proved by noticing that $N$ factors as a product of 
\emph{distinct} irreducible factors $N_1 \cdots N_m$ and, then, by
gluing the corresponding $\taylisom_{N_i}^{\cal C}$ using the Chinese
Remainder Theorem.
In this general setting,
the inverse of $\taylisom_N^{\calC}$ can be easily described: it maps
$T$ to $N$ and $X \in \calC/N\calC$ to the unique root of $N$ in 
$\varprojlim_{m > 0} \calC/N^m\calC$ which is congruent to $X$ modulo $N$. 
The existence and the unicity of this root follows from Hensel's Lemma 
thanks to our assumption that $N$ is separable: it can be obtained as the
limit of the Newton iterative sequence:
$$X_0 = X,
\qquad
X_{i+1} = X_i - \frac{N(X_i)}{N'(X_i)}.$$

\medskip

Of course, the above discussion is still valid when $\calC$ is 
replaced by $\calZ$ (and $K$ is replaced by $F$ accordingly). For
any separable polynomial $F \in \calZ$, we then have constructed
a well defined isomorphism:
$$\taylisom_N^{\calZ} : \,\,
\varprojlim_{m > 0} \calZ/N^m\calZ \,\, 
\simeq \,\, (\calZ/N\calZ)[[T]]$$
We note that $N$ remains separable in $\calC$, implying that 
$\taylisom_N^{\calC}$ is also defined. The unicity property ensures
moreover that the following diagram is commutative:
\begin{equation}
\label{diag:taylorcomm}
\raisebox{2.5em}{%
\xymatrix @C=4em @R=2.5em {
\displaystyle \varprojlim_{m > 0} \calC/N^m\calC \ar[r]^-{\taylisom_N^\calC} 
  & (\calC/N\calC)[[T]] \\
\displaystyle \varprojlim_{m > 0} \calZ/N^m\calZ \ar[r]^-{\taylisom_N^\calZ} \ar@{^(->}[u] 
  & (\calZ/N\calZ)[[T]] \ar@{^(->}[u] }}
\end{equation}
where the vertical arrows are the canonical inclusions.

\subsection{A Taylor-like isomorphism for skew polynomials}

We now aim at completing the diagram~\eqref{diag:taylorcomm} by adding
a top row at the level of Ore rings. For now on, we fix a separable
polynomial $N \in \calZ$. To simplify notations, we set:
$$\hat \calA_N = \varprojlim_{m \geq 1} \calA/N^m\calA 
\quad ; \quad
\hat \calC_N = \varprojlim_{m \geq 1} \calC/N^m\calC
\quad ; \quad
\hat \calZ_N = \varprojlim_{m \geq 1} \calZ/N^m\calZ.$$
Here is our first theorem.

\begin{theo}
\label{theo:taylor}
\begin{enumerate}[(i)]
\item There exists an isomorphism of $K$-algebras
$\taylisom_N : 
\hat \calA_N
\stackrel\sim\longrightarrow 
(\calA/N\calA)[[T]]$ mapping $N$ to $T$ and inducing
the identity of $\calA/N\calA$ after quotienting out
by $N$ of the left and $T$ and the right.

\item Any isomomorphism $\taylisom_N$ satisfying the conditions
of (i) sits in the following commutative diagram:
\begin{equation}
\label{diag:taylor}
\xymatrix @C=4em @R=2.5em {
\hat \calA_N \ar[r]^-{\taylisom_N} 
  & (\calA/N\calA)[[T]] \\
\hat \calC_N \ar[r]^-{\taylisom_N^\calC} \ar@{^(->}[u] 
  & (\calC/N\calC)[[T]] \ar@{^(->}[u] }
\end{equation}
\end{enumerate}
\end{theo}

\begin{rem}
\label{rem:normalized}
If $N$ is an irreducible polynomial in $\calZ$, the polynomials $a 
X^{nr} N$ (with $a \in F$ and $n \in \ZZ$) are also irreducible
in $\calZ$ and they all generate the same ideal. If $\taylisom_N$
satisfies the conditions of Theorem~\ref{theo:taylor}, then a
suitable choice for $\taylisom_{aX^{nr} N}$ is $\iota \circ \taylisom_N$
where $\iota$ is the automorphism of $(\calA/N\calA)[[T]]$ taking
$T$ to $a X^{nr} T$.

In what follows, we shall say that a Laurent polynomial $N \in \calZ$ is 
\emph{normalized} if $N \in \calZ^+$, $N$ is monic and $N$ has a 
nonzero constant coefficient. With this definition, any ideal of $\calZ$ 
has a unique normalized generator.
\end{rem}

\begin{proof}[Proof of Theorem~\ref{theo:taylor}]
The general strategy of the proof is inspired by the caracterization
of the inverse of $\taylisom_N$ we gave earlier: we are going to construct 
the inverse of $\taylisom_N$ by finding a root of $N$ in $\hat \calA_N$.
Without loss of generality, we may assume that $N$ is normalized.
Write $N = a_0 + a_1 Y + \cdots + a_d Y^d$ with $a_i \in F$.
For $f \in \calA$, we define:
$$N(f) = a_0 + a_1 f^r + a_2 f^{2r} + \cdots + a_d f^{rd} \in \calA.$$
We also set
$N' = \frac{dN}{dY} = a_1 + 2 a_2 Y + \cdots + d a_d Y^{d-1}$.
In addition, we choose and fix an element $a \in K$ with $\Tr_{K/F}(a) 
= 1$.

As in Hensel's Lemma, we proceed by successive approximations in order
to find a root of $N$ in $\hat \calA_N$. Precisely, we shall construct by 
induction a sequence $(Z_m)_{m \geq 1}$ of polynomials in $\calZ^+$ with 
$Z_1 = 0$, $Z_{m+1} \equiv Z_m \pmod{N^m}$ and $N(X + a Z_m X) \in N^m 
\calZ^+$ for all $m > 1$. In what follows, we will often write $C_m$ for 
$1 + a Z_m \in \calC^+$.
We assume that $Z_m$ has been constructed for some $m \geq 1$.
The second condition we need to fulfill implies that $Z_{m+1}$ must
take the form $Z_{m+1} = Z_m + a N^m Z$ for some $Z \in \calZ^+$. The
third condition then reads $N(C_{m+1} X) \in N^{m+1} \calZ^+$.

Let us first prove that $N(C_{m+1} X)$ lies in $\calZ^+$. For this,
we observe that
$$(C_{m+1} X)^r = 
\big(1 + a Z_{m+1}\big) \cdot
\big(1 + \theta(a) Z_{m+1}\big) \cdots
\big(1 + \theta^{r-1}(a) Z_{m+1}\big) \cdot X^r.$$
The latter is obviously a polynomial in $X^r$ with coefficients 
in $K$. Since it is moreover stable by the action of $\theta$, its
coefficients must lie in $F$ and we have proved that $(C_{m+1} X)^r
\in \calZ^+$. The fact that $N(C_{m+1} X) \in \calZ^+$ follows directly.

It remains now to ensure that $N(C_{m+1} X)$ is divisible by $N^{m+1}$ 
for a suitable choice of $Z$. For any positive integer $n$, we have the 
following sequence of congruences modulo~$N^{m+1}$:
\begin{align*}
(C_{m+1} X)^{rn}
& \equiv (C_m X)^{rn} + \sum_{i=0}^{rn-1} (C_m X)^i a N^m Z X (C_m X)^{rn-1-i} \\
& \equiv (C_m X)^{rn} + \sum_{i=0}^{rn-1} X^i a N^m Z X^{rn-i} 
& \text{since } C_m \equiv 1 \pmod N \\
& \equiv (C_m X)^{rn} + \sum_{i=0}^{rn-1} \theta^i(a) X^{rn} N^m Z \\
& \equiv (C_m X)^{rn} + X^{rn} N^m Z \pmod {N^{m+1}}
& \text{since } \Tr_{K/F}(a) = 1.
\end{align*}
Therefore $N(C_{m+1} X) \equiv N(C_m X) + X^r N' N^m Z \pmod{N^{m+1}}$.
By the induction hypothesis, $N(C_m X) = N^m S$ with $S \in \calZ^+$. We
are then reduced to prove that there exists a polynomial $Z \in \calZ^+$
such that $S + X^r N' Z \equiv 0 \pmod N$, which follows from the fact
that $X^r N'$ is coprime with $N$.

The sequence $(Z_m)_{m \geq 1}$ we have just constructed defines an
element $Z \in \hat \calZ_N$. We set $C = 1 + a Z$; it is an element
of $\hat\calC_N$. Besides, by construction, $CX$ is a root of $N$, in 
the sense that $N(CX) = 0$. This property together with the fact that
$C$ is invertible in $\hat \calC_N$ ensure that the map
$\iota : \calA/N\calA \to \hat \calA_N$, $X \mapsto CX$
is a well defined morphism of $K$-algebras (see also \S 
\ref{ssec:endoOre}). Moreover, since $C \equiv 1 \pmod N$, $\iota$
reduces to the identity modulo $N$. Mapping $T$ to $N$, one extends
$\iota$ to a second morphism of $K$-algebras:
$$\tau : (\calA/N\calA)[[T]] \to \hat \calA_N.$$
The latter induces the identity after reduction modulo $T$ on the
left and $N$ on the right. Since the source and the target are both
separated are complete (for the $T$-adic and the $N$-adic topology
respectively), we conclude that $\tau$ has to be an isomorphism.
We finally define $\taylisom_N = \tau^{-1}$ and observe that it
satisfies all the requirements of the theorem.

It remains to prove (ii).
By Theorem \ref{theo:endoquotOre}, given a positive integer $m$,
any morphism of $K$-algebras
$\calA/N\calA \to \calA/N^m\calA$ takes $\calC/N\calC$ to $\calC/N^m
\calC$. Passing to the limit, we find that any morphism of $K$-algebras
$\calA/N\calA \to \hat\calA_N$ must send $\calC/N\calC$ to $\hat
\calC_N$.
Therefore, any isomorphism $\taylisom_N$ satisfying the conditions of (i)
induces a morphism of $K$-algebras $(\calC/N\calC)[[T]] \to \hat \calC_N$
which continues to map $T$ to $N$ and induces the identity
modulo $T$. By the unicity result in the commutative case, we deduce
that $\taylisom_N$ coincides with $\taylisom_N^\calC$ on $(\calC/N\calC)[[T]]$,
hence (ii).
\end{proof}

\subsubsection{About unicity}

Unfortunately, unlike the commutative case, the isomorphism
$\taylisom_N$ is not uniquely determined by the conditions of 
Theorem~\ref{theo:taylor}. We nevertheless have several results in
this direction.

\begin{prop}
\label{prop:taylorunique}
Let $\taylisom_{N,1},\, \taylisom_{N,2} : \hat \calA_N \to
(\calA/N\calA)[[T]]$ be two isomorphisms of $K$-algebras
satisyfing the conditions of Theorem~\ref{theo:taylor}.
Then, there exists $V \in (\calC/N\calC)[[T]]$ with $V
\equiv 1 \pmod T$ such that 
$\taylisom_{N,1}(f) = V^{-1} \: \taylisom_{N,2}(f) \: V$ 
for all $f \in \hat \calA_N$.
\end{prop}

\begin{proof}
Set $\gamma = \taylisom_{N,2}^{-1} \circ \taylisom_{N,1}$; it is
an endomorphism of $K$-algebras of $\hat \calA_N$. 
Besides, thanks to the unicity result in the 
commutative case, $\taylisom_{N,1}$ and $\taylisom_{N,2}$ have to
coincide on $\hat \calC_N$. This means that $\gamma$ is in fact
a morphism of $\hat \calC_N$-algebras.
Applying Theorem~\ref{theo:endoquotOre} and passing to the limit, 
this implies the existence of an invertible element $U \in \hat 
\calC_N$, $U \equiv 1 \pmod N$ such that $\gamma(f) = U^{-1} f U$ for 
all $f \in \hat \calA_N$.
Applying $\taylisom_{N,2}$ to this equality, we find that the
proposition holds with $V = \taylisom_{N,2}(U)$.
\end{proof}

\begin{cor}
\label{cor:taylorinv}
Given $f \in \calA$ and $N$ as before, the following quantities 
are preserved when changing the isomorphism $\taylisom_N$:
\begin{enumerate}[(i)]
\item the $T$-adic valuation of $\taylisom_N(f)$,
\item[(i')] more generally, for $j \in \ZZ$, the $T$-adic valuation
of $\sec_j(\taylisom_N(f))$,
\item the first nonzero coefficient of $\taylisom_N(f)$,
\item[(ii')] more generally, for $j \in \ZZ$, the first nonzero
coefficient of $\sec_j(\taylisom_N(f))$,
\item the $0$-th section of $\taylisom_N(f)$, namely $\sec_0 (
\taylisom_N(f))$,
\item[(iii')] more generally, any quantity of the form
$\sigma_{j_1, \ldots, j_m}(\taylisom_N(f))$ with
$j_1 + \cdots + j_m \equiv 0 \pmod r$.
\end{enumerate}
\end{cor}

\begin{proof}
By Proposition~\ref{prop:taylorunique}, if $\taylisom_{N,1}$ and
$\taylisom_{N,2}$ are two suitable isomorphisms, there exists an
invertible element $V \in (\calC/N\calC)[[T]]$, $V \equiv 1 \pmod
T$ such that:
\begin{equation}
\label{eq:taylorunique}
\taylisom_{N,1}(f) = V^{-1} \cdot \taylisom_{N,2}(f) \cdot V.
\end{equation}
The items (i) and (ii) follows. 
Let $j \in \ZZ$. By Lemma~\ref{lem:form:sec}, applying $\sec_j$ to 
\eqref{eq:taylorunique}, we get:
$$\sigma_j \circ \taylisom_{N,1}(f) = V^{-1} \cdot 
\sigma_j \circ \taylisom_{N,2}(f) \cdot \theta^j(V)$$
which implies (i') and (ii'). Finally (iii) and (iii') follow from
Proposition~\ref{prop:endoinv}.
\end{proof}

When $p$ does not divide $r$, the situation is even better because
one can select a canonical representative for $\taylisom_N$. Precisely,
we have the following theorem.

\begin{theo}
\label{theo:taylorcan}
We assume that $p$ does not divide $r$. 
\begin{enumerate}[(i)]
\item 
The homomorphism of $K$-algebras:
$$\taylisom_{N,\can} : \hat \calA_N \to (\calA/N\calA)[[T]],
\quad
X \mapsto \left(\frac{\taylisom_N^\calC(Y)} Y\right)^{\!1/r} \cdot X$$
satisfies the conditions of Theorem~\ref{theo:taylor}.

\item 
The morphism $\taylisom_{N,\can}$ is the unique isomorphism $\taylisom_N : 
\hat \calA_N \to (\calA/N\calA)[[T]]$ which satisfies the conditions of 
Theorem~\ref{theo:taylor} and the extra property $\taylisom_N(X) \in 
(\calZ/N\calZ)[[T]] \cdot X$.
\end{enumerate}
\end{theo}

\begin{rem}
Note that $\taylisom_N^\calC(Y)$ is an element of $\calZ$ which is
congruent to $Y$ modulo $T$. Therefore $\frac{\taylisom_N^\calC(Y)} Y$
is congruent to $1$ modulo $T$ and raising it to the power $\frac 1 r$
makes sense in $(\calZ/N\calZ)[[T]]$ thanks to the formula:
$$\big(1 + xT\big)^{1/r} = \sum_{n=0}^\infty  \,\,
\underbrace{
\frac 1{n!} \cdot \frac 1 r \cdot \left( \frac 1 r - 1 \right)
\cdots \left( \frac 1 r - (n{-}1) \right)}_{c_n} {}\cdot x^n \: T^n.$$
Observe that all the coefficients $c_n$'s lie in $\ZZ[\frac 1 r]$ and 
so can be safely reduced modulo $p$ if $p$ does not divide $r$.
\end{rem}

\begin{proof}[Proof of Theorem~\ref{theo:taylorcan}]
The first part of the theorem is easily checked.
We now assume that we are given two isomorphisms of $K$-algebras
$\taylisom_{N,1},\, \taylisom_{N,2} : \hat \calA_N \to
(\calA/N\calA)[[T]]$ satisfying the conditions of the theorem.
For $i \in \{ 1, 2\}$, we write $\taylisom_{N,i}(X) = Z_i X$ with
$Z_i \in (\calZ/N\calZ)[[T]]$.
By Proposition \ref{prop:taylorunique}, we know that these exists 
$V \in (\calC/N\calC)[[T]]$ such that $V \equiv 1 \pmod T$ and
$$V \cdot \taylisom_{N,1}(f) = \taylisom_{N,2}(f) \cdot V$$
for all $f \in \hat \calA_N$. In particular, for $f = X$, we get
$V Z_1 X = Z_2 X V$, implying $V Z_1 = \theta(V) Z_2$ in 
$(\calC/N\calC)[[T]]$. Taking
the trace from $K$ to $F$, we end up with $W Z_1 = W Z_2$ 
with $W = V + \theta(V) + \cdots + \theta^{r-1}(V)$. Observe 
that $W \equiv r \pmod T$; therefore, it is invertible in 
$(\calZ/N\calZ)[[T]]$ and the equality $W Z_1 = W Z_2$ readily
implies $Z_1 = Z_2$, that is $\taylisom_{N,1} = \taylisom_{N,2}$.
\end{proof}

\subsection{Taylor expansions of skew rational functions}
\label{sssec:taylorexpansions}

Recall that we have defined in \S\ref{ssec:frac} the 
fraction field $\Frac(\calA)$ of $\calA$ and we have proved 
that $\Frac(\calA) = \Frac(\calZ) \otimes_\calZ \calA$ (see 
Theorem~\ref{theo:fracA}).

\subsubsection{Taylor expansion at central separable polynomials}

For a given separable polynomial $N \in \calZ$,
the isomorphism $\taylisom_N$ of Theorem~\ref{theo:taylor} extends 
to an isomorphism $\Frac(\calZ) \otimes_\calZ \hat \calA_N \to
(\calA/N\calA)(\!(T)\!)$
and we can consider the composite:
$$\taylor_N : \Frac(\calA) = \Frac(\calZ) \otimes_\calZ \calA 
\longrightarrow \Frac(\calZ) \otimes_\calZ \hat \calA_N 
\stackrel\sim\longrightarrow (\calA/N\calA)(\!(T)\!)$$
where the first map is induced by the natural inclusion $\calA
\to \hat\calA_N$. By definition $\taylor_N(f)$ is called the 
\emph{Taylor expansion} of $f$ around $N$. We notice that it
does depend on a choice of the isomorphism $\taylisom_N$. However, 
one can form several quantities that are independant of any choice 
and then are canonically attached to $f \in \Frac(\calA)$ and $N$
as before. 
Many of them are actually given by Corollary~\ref{cor:taylorinv};
here are they:
\begin{enumerate}[(i)]
\item the \emph{order of vanishing} of $f$ at $N$, denoted by 
$\ord_N(f)$; it is defined as the $T$-adic valuation of $\taylor_N(f)$,
\item[(i')] for $j \in \ZZ$, the \emph{$j$-th partial order of 
vanishing} of $f$ at $N$, denoted by $\ord_{N,j}(f)$; it is defined as 
the $T$-adic valuation of $\sec_j(\taylor_N(f))$,
\item the \emph{principal part} of $f$ at $N$, denoted by $\pp_N(f)$; it is
the element of $\calA/N\calA$ 
defined as the coefficient of $T^{\ord_N(f)}$ in the series $\taylor_N(f)$,
\item[(ii')] for $j \in \ZZ$, the \emph{$j$-th partial principal part} of 
$f$ at $N$, denoted by $\pp_{N,j}(f)$; it is the element of $\calC/N\calC$
defined as the coefficient of 
$T^{\ord_{N,j}(f)}$ in the series $\sec_j(\taylor_N(f))$,
\item the $0$-th section of $\taylor_N(f)$, namely $\sec_0(\taylor_N(f))$,
\item[(iii')] more generally, any quantity of the form
$\sigma_{j_1, \ldots, j_m}(\taylor_N(f))$ with
$j_1 + \cdots + j_m \equiv 0 \pmod r$.
\end{enumerate}
The previous invariants are related by many relations, \emph{e.g.}:
\begin{itemize}
\item $\ord_N(f) = \min\big(\ord_{N,0}(f), \, \ldots, \ord_{N,r-1}(f)\big)$,
\item $\ord_{N,j+r}(f) = \ord_{N,j}(f)$,
\item $\pp_N(f) = \sum_j \pp_{N,j}(f) \: X^j$
where the sum is extended to the indices $j \in \{0, 1, \ldots, r{-}1\}$ 
for which $\ord_{N,j}(f) = \ord_N(f)$,
\item $\pp_{N,j+r}(f) = X^r \: \pp_{N,j}(f)$,
\item $\ord_N(fg) \geq \ord_N(f) + \ord_N(g)$ and equality holds as
soon as $\calA/N\calA$ is a division algebra\footnote{This is the case 
for instance if $K = \CC$, $\theta$ is the complex conjugacy and 
$N = X^2 + z$ with $z \in \RR_{>0}$.},
\item $\pp_N(fg) = \pp_N(f) \cdot \pp_N(g)$ when
$\ord_N(fg) = \ord_N(f) + \ord_N(g)$.
\end{itemize}
We say that $f$ has \emph{no pole} at $N$ when $\ord_N(f) \geq 0$. 
It has a \emph{simple pole} at $N$ when $\ord_N(f) = -1$. Generally,
we define the order of the pole of $f$ at $N$ as the opposite of
$\ord_N(f)$.

\subsubsection{Taylor expansion at nonzero closed points}

In a similar fashion, one can define the Taylor expansion of a skew 
rational function at a nonzero closed point $z$ of $F$. When $z$ is 
rational, \emph{i.e.} $z \in F$, $z \neq 0$, we simply set $\taylor_z = 
\taylor_{Y-z}$.

Otherwise, the construction is a bit more subtle.
Let $\Fsep$ be a fixed separable closure of $F$ and let $z \in 
\Fsep$, $z \neq 0$. Let also $N \in \calZ^+$ be the minimal 
polynomial of $z$. We have recalled in \S \ref{sssec:taylorreminders}
that the Taylor expansion around $z$ defines an isomorphism:
$$\taylisom_z^\calC : 
\hat \calC_N \stackrel\sim\longrightarrow
(\calC/N\calC)[[T]]$$
which is characterized by the fact that it sends $Y$ to $z+T$.
In general, $\taylisom_z^\calC$ does not agree with $\taylisom_N^\calC$ 
but there exists a series $S_z \in (\calC/N\calC)[[T]]$ such that
$\taylisom_z^\calC = \varphi_z \circ \taylisom_N^\calC$
where $\varphi_z$ is the endomorphism of $(\calC/N\calC)[[T]]$ taking 
$T$ to $S_z$ (and acting trivially on the coefficients). The latter
extends to an endomorphism of $(\calA/N\calA)[[T]]$, that we continue
to call $\varphi_z$. By construction, the following diagram is
commutative:
\begin{equation}
\xymatrix @C=5em @R=2.5em {
\hat \calA_N \ar[r]^-{\varphi_z \circ \taylisom_N} 
  & (\calA/N\calA)[[T]] \\
\hat \calC_N \ar[r]^-{\taylisom_z^\calC} \ar@{^(->}[u] 
  & (\calC/N\calC)[[T]] \ar@{^(->}[u] }
\end{equation}
whenever $\taylisom_N : \hat \calA_N \to
(\calA/N\calA)[[T]]$ is an isomorphism satisfying the conditions
of Theorem~\ref{theo:taylor}.(i).
It worths noticing that the morphisms of the form $\varphi_z \circ
\taylisom_N$ can be characterized without any reference of $\taylisom_N$.

\begin{prop}
\label{prop:taylisomz}
Given $z \in \Fsep$, $z \neq 0$, we have the following equivalence: a
mapping $\taylisom_z : \hat\calA_N \to (\calA/N\calA)[[T]]$ is of the form 
$\varphi_z \circ \taylisom_N$ (where $\taylisom_N$ satisfies the 
condition of Theorem~\ref{theo:taylor}) if and only if
$\taylisom_z$ is a morphism of $K$-algebras,
$\taylisom_z(X) \equiv X \pmod T$ and
$\taylisom_z(Y) = z + T$.
\end{prop}

\begin{proof}
If $\taylisom_z = \varphi_z \circ \taylisom_N$, it follows from the 
conditions of Theorem~\ref{theo:taylor} that $\taylisom_z$ is morphism of 
$K$-algebras which induces the identity modulo $T$. Hence $\taylisom_z(X) 
\equiv X \pmod T$. Moreover, by the second part of 
Theorem~\ref{theo:taylor}, we know that $\taylisom_N$ coincides with 
$\taylisom_N^\calC$ on $\hat\calC_n$. Therefore $\taylisom_z$ has to 
agree with $\varphi_z \circ \taylisom_N^\calC = \taylisom_z^\calC$ on 
$\hat\calC_N$, implying in particular that $\taylisom_z(Y) = z + T$.

Conversely, let us assume that $\taylisom_z$ satisfies the condition
of the proposition. We have to check that $\taylisom_N = \varphi_z^{-1} 
\circ \taylisom_z$ satisfies the conditions of Theorem~\ref{theo:taylor}.
The fact that $\taylisom_N$ is a morphism of $K$-algebras is obvious.
The assumption $\taylisom_z(X) \equiv X \pmod T$ ensures that
$\taylisom_N$ acts as the identity modulo $T$. Finally, the hypothesis
$\taylisom_z(Y) = z + T$ implies that $\taylisom_z$ coincides with
$\taylisom_z^\calC$ on $\hat\calC_N$. Hence:
$$\taylisom_N(N) = \varphi_z^{-1} \circ \taylisom_z(N) =
\varphi_z^{-1} \circ \taylisom_z^\calC(N) = \taylisom_N^\calC(N) = T$$
and we are done.
\end{proof}

\begin{defi}
\label{def:admissible}
Given $z \in \Fsep$, $z \neq 0$ as before,
we say that a morphism $\tau : \hat\calA_N \to (\calA/N\calA)[[T]]$
is \emph{$z$-admissible} if it satisfies the conditions of
Proposition~\ref{prop:taylisomz}.
\end{defi}

\begin{rem}
\label{rem:taylisomz}
By Theorem~\ref{theo:endoquotOre}, a homomorphism of $K$-algebras
$\taylisom_z : \hat\calA_N \to (\calA/N\calA)[[T]]$ is entirely
determined by the element $C = \taylisom_z(X) \: X^{-1} \in
(\calC/N\calC)[[T]]$. Proposition~\ref{prop:taylisomz} shows
that $\tau_z$ is $z$-admissible if and only if:
$$C \equiv 1 \pmod T
\quad \text{and} \quad
\norm_{(\calC/N\calC)[[T]]/(\calZ/N\calZ)[[T]]}\big(C\big)
= 1 + \frac T z.$$
Moreover any $C \in (\calC/N\calC)[[T]]$ satisfying the above
conditions gives rise to an admissible morphism~$\taylisom_z$.
\end{rem}

From now on, we fix a choice of an $z$-admissible morphism $\taylisom_z$.
Accordingly, we define $\taylor_z$ as the composite:
$$\taylor_z : \Frac(\calA) = \Frac(\calZ) \otimes_\calZ \calA 
\longrightarrow \Frac(\calZ) \otimes_\calZ \hat \calA_N 
\stackrel{\taylisom_z}\longrightarrow (\calA/N\calA)(\!(T)\!).$$
Like $\taylor_N$, the morphism $\taylor_z$ depends upon some
choices but some quantities attached to it are canonical, as the
order of vanishing at $z$, the principal part at $z$, \emph{etc.}
For $f \in \Frac(\calA)$ and $j \in \ZZ$, we use the transparent
notations $\ord_z(f)$, $\ord_{z,j}(f)$, $\pp_z(f)$ and $\pp_{z,j}(f)$
to refer to them.

\begin{prop}
Let $z \in \Fsep$, $z \neq 0$ and let $N \in \calZ^+$ be
its minimal polynomial. Then:
\begin{enumerate}[(i)]
\item $\ord_z(f) = \ord_N(f)$,
\item[(i')] $\ord_{z,j}(f) = \ord_{N,j}(f)$ for all $j \in \ZZ$,
\item $\pp_z(f) = \pp_N(f)$,
\item[(ii')] $\pp_{z,j}(f) = \pp_{N,j}(f)$ for all $j \in \ZZ$.
\end{enumerate}
\end{prop}

\begin{proof}
Everything follows from the facts that $\varphi_z$ preserves the
valuation, the principal part and commutes with $\sec_j$.
\end{proof}

\subsubsection{Taylor expansion at $0$}
\label{sssec:taylor0}

Until now, we have always paid attention to exclude the special point 
$z = 0$. Indeed, when $z = 0$, the situation is a bit different 
because, roughly speaking, the ideal $(Y)$ ramifies in the extension 
$\calA^+ / \calC^+$.
However, it is also possible (and even simpler) to define Taylor
expansions around $0$. In order to do this, we first define:
$$\hat\calA^+_0 = \varprojlim_{m > 0} \calA^+/Y^m \calA^+
\qquad \text{and} \qquad
\textstyle \hat\calA_0 = \hat\calA^+_0[\frac 1 Y].$$
The elements of $\hat\calA^+_0$ can be viewed as power series in 
the variable $X$, that is infinite sums of the form:
$$f = a_0 + a_1 X + \cdots + a_n X^n + \cdots$$
where the coefficients $a_i$ lie in $K$.
The multiplication on $\hat\calA_0$ is driven by Ore's rule
$X \cdot c = \theta(c) X$ for $c \in K$.
Similarly, the elements of $\hat\calA_0$ are Laurent series of the 
form:
$$f = a_v X^v + a_{v+1} X^{v+1} + \cdots +
a_0 + a_1 X + \cdots + a_n X^n + \cdots$$
where $v$ is a (possibly negative) 
integer and the $a_i$'s are elements of $K$. For
this reason, we will sometimes
write $K(\!(X; \theta)\!)$ instead of $\hat
\calA_0$. Noticing that $\Frac(\calZ)$ canonically embeds into
$F(\!(Y)\!) \subset K(\!(X; \theta)\!)$, we deduce that 
$\Frac(\calZ) \otimes_{\calZ^+} \hat \calA^+_0 \simeq
K(\!(X; \theta)\!)$.
We are now ready the define the Taylor expansion at $0$, following
the construction of $\taylor_N$. We set:
$$\taylor_0 : \Frac(\calA) = \Frac(\calZ) \otimes_{\calZ^+} \calA^+
\longrightarrow \Frac(\calZ) \otimes_{\calZ^+} \hat \calA^+_0
\stackrel\sim\longrightarrow K(\!(X; \theta)\!).$$
Unlike $\taylor_z$, the morphism $\taylor_0$ is entirely canonical
and does not depend upon any choice.

\subsubsection{Taylor expansion and derivations}

In the commutative case, the coefficients of the Taylor expansion of a 
function $f$ around one rational point $z$ are given by the values at 
$z$ of the successive divided derivatives of $f$ (see
Eq.~\eqref{eq:taylorclassical}). Below, we will establish similar 
results in the noncommutative setting.

We consider an element $z \in \Fsep$, $z \neq 0$.
Let $N \in \calZ^+$ be the minimal polynomial of~$z$.
Let $\taylisom_z : 
\hat \calA_N \to (\calA/N\calA)[[T]]$ be any $z$-admissible morphism 
(see Definition~\ref{def:admissible}). We define $C = \taylisom_z(X) 
X^{-1} \in (\calC/N\calC) [[T]]$. It is congruent to $1$ modulo~$T$; 
in particular, it is invertible in $(\calC/N\calC) [[T]]$.
The codomain of $\taylisom_z$, namely $(\calA/N\calA)[[T]]$, is 
canonically endowed with the 
derivation $\frac d{dT}$. A simple computation shows that it
corresponds to the derivation $\partial_{\derC}$ on
$\hat \calA_N$ where $\derC$ is defined by:
$$\derC = \taylisom_z^{-1}\left( C^{-1} \: \frac{dC}{dT} \right)
\in \hat \calA_N.$$
The $p$-th power of $\partial_{\derC}$ vanishes since it corresponds
to $\frac{d^p}{dT^p}$ through the isomorphism $\taylisom_z$. Using
$\taylisom_z$, we can go further and define higher divided powers of 
$\partial_{\mathfrak C}$ by:
\begin{equation}
\label{eq:dividedpow}
\partial_{\derC}^{[n]} = \taylisom_z^{-1} \circ \left(\frac 1{n!}\:
\frac {d^n}{dT^n}\right) \circ \taylisom_z
\end{equation}
for all nonnegative integer $n$. With this definition, it is 
formal to check that:
\begin{equation}
\label{eq:taylorOre}
\taylisom_z(f) = \sum_{n=0}^\infty \partial_{\derC}^{[n]}(f) \cdot T^n
\,\,\in\,\,(\calA/N\calA)[[T]].
\end{equation}
However, this result does not give much information because $\derC$
is hard to compute (and the $\partial_{\derC}^{[n]}$'s are even
harder) and depends heavily on $z$.
Typically, Proposition~\ref{prop:pthpowerzero} shows that 
$\derC$ cannot be rational unless $r$ is coprime with $p$.
Nevertheless, when $p$ does not divide $r$ and $\taylisom_z$ is well 
chosen, we shall see that the computation of $\derC$ and 
$\partial_{\derC}^{[n]}$ can be carried out and yields eventually
a simple interpretation of the Taylor coefficients.

Frow now on, we assume that $p$ does not divide $r$. By 
Theorem~\ref{theo:taylorcan}, we know that there is a canonical choice 
for $\taylisom_z$, called $\taylisom_{z,\can}$. The corresponding 
element $C$ is: 
$$C_\can = \left(\frac{\taylisom_z^\calC(Y)}Y\right)^{\!1/r} = 
\left(1 + \frac T z\right)^{\!1/r}.$$ 
Therefore: 
$$\mathfrak C_\can = 
\taylisom_z^{-1}\left(C_\can^{-1} \: \frac{d C_\can}{dT}\right) = 
\taylisom_z^{-1}\left(\frac 1 r \: \frac 1 {T+z}\right) = \frac 
1{rY}.$$ 
In particular, we observe that $\derC_\can$ is rational and, 
even better, $\partial_{\derC_\can}$ is equal to the canonical 
derivative $\partial_\can$ we introduced in Definition~\ref{defi:dercan}. 
Its divided powers (defined by Eq.~\eqref{eq:dividedpow}) also have 
a simple expression: 
$$\partial_\can^{[n]}\Big(\sum_i a_i X^i\Big) = 
\sum_i\,\, \underbrace{\frac 1{n!} \cdot \frac i r \cdot \left( 
\frac i r - 1 \right) \cdots \left( \frac i r - (n{-}1) \right)}_{c_{n,i}} 
{}\cdot a_i \: X^{i-rn}.$$
where the coefficients $c_{n,i}$'s all lie in $\ZZ[\frac 1 
r]$ and, consequently, can be reduced modulo $p$ without trouble. With 
these inputs, Eq.~\eqref{eq:taylorOre} reads: 
\begin{equation} 
\label{eq:taylorderiv} 
\taylisom_{z,\can}(f) = \sum_{n=0}^\infty 
\partial_\can^{[n]}(f) \: T^n \,\,\in\,\,(\calA/N\calA)[[T]] 
\end{equation} 
which can be considered as a satisfying skew analogue of the
classical Taylor expansion formula.

\section{A theory of residues}
\label{sec:residue}

The results of the previous section lay the foundations of a theory of 
residues for skew polynomials. The aim of the present section is to 
develop it: we define a notion of residue at a closed point of $F$ 
for skew rational functions and then prove the residue formula and 
study how residues behave under change of variables.

Throughout this subsection, we fix a separable closure $\Fsep$ of $F$,
together with an embedding $K \hookrightarrow \Fsep$.
For $z \in \Fsep$ and $C \in \Frac(\calC)$, we will write 
$\res_z(C {\cdot} dY)$ for the (classical) residue at $z$ of the
differential form $C{\cdot} dY$.

\subsection{Definition and first properties}

We recall that, for $z \in \Fsep$, $z \neq 0$, we have defined 
in \S \ref{sssec:taylorexpansions} a (non canonical) morphism of 
$K$-algebras:
$$\taylor_z : \Frac(\calA) \longrightarrow (\calA/N\calA)(\!(T)\!)$$
where $N \in \calZ^+$ is the minimal polynomial of $z$.
On the other hand, there is a natural embedding $\calZ/N\calZ 
\hookrightarrow \Fsep$ obtained by mapping $Y$ to $z$. 
Extending scalars from $F$ to $K$, it extends to a second embedding
$$\iota_z : \calC/N\calC \longrightarrow K \otimes_F \Fsep.$$
We observe that the codomain of $\iota_z$,
namely $K \otimes_F \Fsep$, is naturally isomorphic to a product
of $r$ copies of $\Fsep$ \emph{via} the mapping:
$$\beta: K \otimes_F \Fsep \to (\Fsep)^r, \qquad
c \otimes x \mapsto \big(cx, \theta(c) x, \ldots, \theta^{r-1}(c) 
x\big).$$

\begin{defi}
\label{def:residue}
For $z \in \Fsep$, $z \neq 0$, and $f \in \Frac(\calA$), we
define:
\begin{itemize}
\item the \emph{skew residue} of $f$ at $z$, denoted by $\sres_z(f)$,
as the coefficient of $T^{-1}$ in the series $\taylor_z(f)$; it is 
an element of $\calA/N\calA$,
\item for $j \in \{0, \ldots, r{-}1\}$, the \emph{$j$-th partial skew 
residue} of $f$ at $z$, denoted by $\sres_{z,j}(f)$, as:
$$\iota_z \circ \sec_j \circ \sres_z(f) \in \big(K \otimes_F \Fsep\big).$$
\end{itemize}
\end{defi}

Here are two important remarks concerning residues.
First, we insist on the fact that both $\sres_z(f)$ and $\sres_{z,j}(f)$ 
do depend on the choice of the $z$-admissible morphism $\taylisom_z$ (used in 
the definition of $\taylor_z$) in general. However, 
Corollary~\ref{cor:taylorinv} shows that $\sres_z(f)$ and $\sres_{z,j}(f)$ 
are defined without ambiguity when $f$ has (at most) a simple pole at 
$z$. Besides, when $p$ does not divide $r$, there is a distinguished 
choice for $\taylor_z$ (see Theorem~\ref{theo:taylorcan}), leading to 
distinguished choices for $\sres_z$ and $\sres_{z,j}$. In the sequel, we 
will denote them by $\sres_{z,\can}$ and $\sres_{z,j,\can}$.

Second, we observe that, the collection of all the partial skew residues 
$\sres_{z,j}(f)$'s captures as much information as $\sres_z(f)$, given 
that $\sres_z(f)$ is determined by its sections $\sec_j(\sres_z(f))$'s 
with $0 \leq j < r$ thanks to the formula: 
$$\sres_z(f) = \sum_{j=0}^{p-1} \sec_j \circ \sres_z(f).$$

\subsubsection{Residues at special points}

It will be convenient to define residues at $0$ and $\infty$ as well. 
For residues at $0$, we recall that we have defined in 
\S\ref{sssec:taylor0} a \emph{canonical} Taylor expansion map around $0$:
$$\taylor_0 : \Frac(\calA) \longrightarrow K(\!(X;\theta)\!)$$

\begin{defi}
\label{def:residue0}
For $f \in \Frac(\calA)$ and $j \in \{0,1,\ldots,r{-}1\}$, we define the 
\emph{$j$-th partial skew residue} of $f$ at $0$, denoted by 
$\sres_{0,j}(f)$, as the coefficient of $X^{j-r}$
in the series $\taylor_0(f)$.
\end{defi}

Residues at infinity are defined in a similar fashion. 
Let $\tilde X$ be a new variable and form the skew algebra $\tilde 
\calA = K[\tilde X^{\pm 1}; \theta^{-1}]$. Clearly $\tilde \calA$ 
is isomorphic to $\calA$ by letting $\tilde X$ correspond to $X^{-1}$.
We then get a map:
$$\taylor_\infty : \Frac(\calA) \simeq \Frac(\tilde \calA) 
\longrightarrow K(\!(\tilde X ; \theta^{-1})\!)$$
where the second map is the morphism $\taylor_0$ for $\tilde \calA$.

\begin{defi}
\label{def:residueinfty}
For $f \in \Frac(\calA)$ and $j \in \{0,1,\ldots,r{-}1\}$, we define 
the \emph{$j$-th partial skew residue} of $f$ at $\infty$, denoted by 
$\sres_{\infty,j}(f)$, as the opposite of the coefficient of 
$\tilde X^{r-j}$ in the series $\taylor_\infty(f)$.
\end{defi}

Unlike $\sres_{z,j}(f)$, the partial skew residues $\sres_{0,j}(f)$ and 
$\sres_{\infty,j}(f)$ do not depend on any choice and so are canonically 
attached to $f$.

\subsubsection{Commutative residues}

The skew residues we just defined are closely related, in many
cases, to classical residues of rational differential forms. In order
to state precise results in this direction, we need extra notations.
We observe that the map $\res_z$ 
defines by restriction an $F$-linear mapping $\calZ \: dY \to 
\Fsep$. Tensoring it by $K$ over $F$, we obtain a $K$-linear
map $\rho_z : \calC \: dY \longrightarrow K \otimes_F \Fsep$.
Letting $\res : (\calC/N\calC)(\!(T)\!) \to \calC/N\calC$ be the
map selecting the coefficient in $T^{-1}$, one checks the two
following formulas:
\begin{align*}
\rho_z\big(C{\cdot}dY\big) 
& = \iota_z \circ \res \circ \taylor_z\big(C\big) \\
\beta \circ \rho_z\big(C{\cdot}dY\big) 
& = 
\big(\res_z\big(C{\cdot}dY\big), \res_z\big(\theta(C){\cdot}dY\big), 
\ldots \res_z\big(\theta^{r-1}(C){\cdot}dY\big)\big)
\end{align*}
for all $C \in \Frac(\calC)$.

\begin{prop}
\label{prop:ressec0}
For $z \in \Fsep \sqcup \{\infty\}$ and $f \in \Frac(\calA)$, we
have $\sres_{z,0}(f) = \rho_z\big(\sec_0(f){\cdot}dY\big)$.
\end{prop}

\begin{proof}
By definition, $\sres_{z,0}(f) = \iota_z \circ \sec_0 \circ \sres_z(f)$.
Applying Lemma~\ref{lem:endosec} and passing to the limit, we find that 
the isomorphism $\taylisom_z$ 
commutes with $\sec_0$. Hence $\sec_0 \circ \sres_z$ is equal to the
compositum:
$$\Frac(\calA) 
\stackrel{\sigma_0}\longrightarrow
\Frac(\calC) 
\stackrel{\taylor_z}\longrightarrow
(\calC/N\calC)(\!(T)\!) 
\stackrel{\res}\longrightarrow
\calC/N\calC.$$
Composing further by $\iota_z$ on the left, we get the proposition.
\end{proof}

Proposition~\ref{prop:ressec0} implies in particular that $\sres_{z,0}(f)$ 
does not depend on any choice and thus is canonically attached to $f$ 
and $z$. According to Corollary~\ref{cor:taylorinv}, there are other 
invariants which are canonically attached to $\sres_z(f)$. A family of 
them consists of the $\sec_{j_1, \ldots, j_m}(\sres_z(f))$'s
for $j_1, \ldots, j_m \in \ZZ$ with $j_1 + \cdots + j_m \equiv
0 \pmod r$. However, these invariants seem less interesting; for
example, they do not define additive functions on $\Frac(\calA)$.

Under some additional assumptions, other partial skew residues are 
also related to residues of rational differential forms.

\begin{prop}
\label{prop:ressecj}
Let $z \in \Fsep \sqcup \{\infty\}$, let $f \in \Frac(\calA)$ and
let $j \in \{0, 1, \ldots, r{-}1\}$.

\noindent
If $z \in \{0, \infty\}$ or $\ord_{z,j}(f) \geq -1$, then:
$$\sres_{z,j}(f) = \rho_z\big(\sec_j(f){\cdot}dY\big).$$
\end{prop}

\begin{proof}
When $z \in \{0, \infty\}$, the proposition can be easily checked 
by hand. Let us now assume that $\ord_{z,j}(f) \geq -1$.
By Lemma~\ref{lem:endosec}, we know that 
$\sec_j \circ \taylisom_z =
\norm_j(C) \cdot (\taylisom_z \circ \sec_j)$
with $C = \taylisom_z(X) X^{-1} \in (\calC/N\calC)[[T]]$. Moreover,
from the fact that $\taylisom_z$ induces the identity modulo $N$,
we deduce that $C \equiv 1 \pmod T$. Consequently $\taylisom_z$
commutes with $\sec_j$ modulo $T$. The end of the proof is now
similar to that of Proposition~\ref{prop:ressec0}.
\end{proof}

\subsection{The residue formula}

In the classical commutative setting, the theory of residues is very 
powerful because we have at our disposal many formulas, allowing for a 
complete toolbox for manipulating them easily and efficiently. We now
strive to establish analogues of these formulas in our noncommutative 
setting. We start by the ``commutative'' residue formula.

\begin{theo}
\label{theo:commutativeresidue}
For $f \in \Frac(\calA)$, we have:
$$\sum_{z \in \Fsep \sqcup \{\infty\}} \sres_{z,0}(f) = 0.$$
\end{theo}

\begin{proof}
Since $\beta$ is an isomorphism, it is enough to prove that
$\sum_{z \in \Fsep \sqcup \{\infty\}} \beta \circ \sres_{z,0}(f) = 0$.
Writing $C = \sec_0(f) \in \calC$, Proposition~\ref{prop:ressec0} asserts that:
$$\beta \circ \sres_{z,0}(f) 
= \beta \circ \rho_z \big(C\big)
= \big( \res_z\big(C{\cdot}dY\big), \res_z\big(\theta(C){\cdot}dY\big), 
\ldots,
\res_z\big(\theta^{r-1}(C){\cdot}dY\big)\big)$$
in $(\Fsep)^r$.
The theorem them follows from the classical residue formula
applied to the $\theta^j(C)$'s for $j$ varying
between $0$ and $r{-}1$.
\end{proof}

The reader might be a bit disappointed by the previous theorem as it 
only concerns $0$-th partial skew residues and it reduces immediately to 
the classical setting.
Unfortunately, in general, it seems difficult to obtain a vanishing 
result involving skew residues since the latter might be not 
canonically defined. There is however an important special case for 
which such a formula exists and can be proved.

\begin{theo}
\label{theo:residue}
Let $f \in \Frac(\calA)$. We assume that $f$ has at most a simple 
pole at all points $z \in \Fsep$, $z \neq 0$. Then:
$$\sum_{z \in \Fsep \sqcup \{\infty\}} 
\sres_{z,j}(f) = 0$$
for all $j \in \{0, 1, \ldots, r{-}1\}$.
\end{theo}

\begin{proof}
Let $j \in \{0, \ldots, r{-}1\}$ and set $C_j = \sec_j(f)$.
By Proposition~\ref{prop:ressecj}, we know that:
$$\beta \circ \sres_{z,j}(f) 
= \beta \circ \rho_z (C_j)
= \big( \res_z\big(C_j{\cdot}dY\big), \res_z\big(\theta(C_j{\cdot}dY\big)), 
\ldots, \res_z\big(\theta^{r-1}(C_j){\cdot}dY\big)\big)$$
By the classical residue formula applied successively with $C_j, 
\theta(C_j), \ldots,
\theta^{r-1}(C_j)$, we deduce that $\sres_{z,j}(f)$ has to
vanish.
\end{proof}

The case of canonical residues also deserves some attention.
As before, the main input is a formula relating the partial skew
residues $\sres_{z,j,\can}(f)$ to classical residues. 
We consider a new variable $y$ and form the \emph{commutative} 
polynomial ring $K[y]$ and its field of fractions $K(y)$. We embed 
$\Frac(\calC)$ into $K(y)$ by taking $Y$ into $y^r$. We insist on
the fact that $y$ is not $X$ or, equivalently, $K(y)$ is not 
$\Frac(\calA)$: our new variable $y$ commutes with the scalars.
Since $K(y)$ is a genuine field of rational functions, it carries a 
well-defined notion of residue. For $f \in K(y)$ and $z \in \Fsep$, we 
will denote by $\res_z(f{\cdot}dy)$ the residue at $f$ of the 
differentiel form $f{\cdot}dy$.
Similarly the map $\rho_z$ extends to $K(y) \: dy$. 
Performing the change of variable $y \mapsto Y = y^r$, we obtain
the relations:
\begin{align*}
\res_{z^r}\big(C \cdot dY\big) 
& = r \cdot \res_z\big(y^{r-1}\: C \cdot dy\big) \\
\rho_{z^r}\big(C \cdot dY\big) 
& = r \cdot \rho_z\big(y^{r-1}\: C \cdot dy\big)
\end{align*}
which hold true for any $C \in \calC$ and any $z \in \Fsep$.

\begin{prop}
\label{prop:rescansecj}
We assume that $p$ does not divide $r$.

\noindent
For $f \in \Frac(\calA)$, $j \in \{0, 1, \ldots, r{-}1\}$ and
$z \in \Fsep$, $z \neq 0$, we have:
$$\sres_{z, j, \can}(f)
= r \: \zeta^{-j}\: \rho_\zeta\big(y^{j+r-1} \: \sec_j(f)\cdot dy \big)$$
where $\zeta$ is any $r$-th root of $z$.
\end{prop}

\begin{proof}
Set $C_\can = \taylisom_{z,\can}(X) \: X^{-1}$.
From Lemma~\ref{lem:endosec}, we know that:
\begin{equation}
\label{eq:commcan}
\sec_j \circ \taylisom_{z,\can} =
\norm_j(C_\can) \cdot (\taylisom_{z,\can} \circ \sec_j).
\end{equation}
On the other hand, it follows from
Theorem~\ref{theo:taylorcan} that $C_\can \in 
(\calZ/N\calZ)[[T]]$. Since moreover $C_\can \equiv 1 \pmod T$,
writing $\taylisom_{z,\can}(Y) = z + T$, we find
$C_\can = \big(1 + \frac T z\big)^{\!1/r}$.
Plugging this in \eqref{eq:commcan}, we obtain:
\begin{equation}
\label{eq:commcan2}
\sec_j \circ \taylisom_{z,\can} =
\left(1 + \frac T z\right)^{\!j/r} \cdot (\taylisom_{z,\can} \circ \sec_j).
\end{equation}
The main observation is that the twisting function
$\left(1 + \frac T z\right)^{j/r}$ which is \emph{a priori} only
defined on a formal neighborhood of $T = 0$ (or, equivalenty of
$Y = z$) is closely related to a function of the variable $y$ which 
is globally defined.
Precisely, consider the local parameter $t = y - \zeta$ on a formal
neighborhood of $\zeta$. The relation $y^r = Y$ translates to
$(\zeta + t)^r = z + T$. Dividing by $z$ on both sides and raising
to the power $\frac j r$, we obtain:
$$\zeta^{-j} y^j = \left(1 + \frac t \zeta\right)^j = 
\left(1 + \frac T z\right)^{\!j/r}$$
showing that our multiplier $\left(1 + \frac T z\right)^{j/r}$ is
the Taylor expansion of the function $\zeta^{-j} y^j$. 
Eq.~\eqref{eq:commcan2} then becomes
$\sec_j \big(\taylisom_{z,\can}(f)\big) =
\taylisom_{z,\can}\big( \zeta^{-j} y^j \: \sec_j(f)\big)$.
Taking the coefficient in $T^{-1}$, we get:
$$\sres_{z,j,\can}(f) =
\rho_z\big( \zeta^{-j} y^j \cdot \sec_j(f) \cdot dY \big) =
r \cdot \rho_\zeta\big( \zeta^{-j} y^{j+r-1} \: \sec_j(f) \cdot dy \big)$$
which is exactly the formula in the statement of the proposition.
\end{proof}

Unfortunately, Proposition~\ref{prop:rescansecj} does not give an
interesting vanishing result for canonical partial skew residues.
Indeed, if we apply the residue formula to the differential form
$y^{j+r-1} \: \sec_j(f) {\cdot} dy$, we end up with:
\begin{equation}
\label{eq:noninteresting}
\sum_{\substack{\zeta \in \Fsep\\ \zeta\neq 0}}
\zeta^j \cdot \sres_{\zeta^r,j,\can}(f) \,=\, 0.
\end{equation}
Actually, this formula does not give any information because the
sum on the left hand side can be refactored as follows:
$$\sum_{\substack{z \in \Fsep\\ z\neq 0}} 
\Bigg(\sum_{\zeta^r = z} \zeta^j
\cdot \sres_{\zeta^r,j,\can}(f) \Bigg)$$
and each internal sum vanishes simply because $\sum_{\zeta^r = z}
\zeta^j = 0$. In other words, the formula~\eqref{eq:noninteresting}
holds equally true when $\sres_{\zeta^r,j,\can}(f)$ is replaced by any
quantity depending only on $\zeta^r$.

However, Proposition~\ref{prop:rescansecj} remains interesting for 
itself and can even be used to derive relations on partial skew residues 
of a skew rational function $f$. One way to achieve this goes as follows. 
Let $f \in \Frac(\calA)$ and $j \in \{1, \ldots, r{-}1\}$. We assume 
that we know a finite set $\Pi = \{z_1, \ldots, z_n\}$ containing the 
points $z \in \Fsep$, $z \neq 0$ for which $\ord_{z,j}(f) < 0$. We 
assume further, for each index $i$, we are given an integer $n_i$ with 
the guarantee that $\ord_{z,j}(f) \geq -n_i$.
For each $i$, we choose a $r$-th root $\zeta_i$ of $z_i$. Let $P \in 
\Fsep[y]$ be a polynomial such that, for all $i$, $P(\zeta_i) = 
\zeta_i^{-j}$ and the derivative $P'(y)$ has a zero of order at least 
$(n_i-1)$ at $\zeta_i$. This choice of $P$ ensures that:
$$\rho_{\zeta_i}\big(P(y) \: y^{j+r-1} \: \sec_j(f)\cdot dy \big)
= \zeta_i^{-j}\: \rho_{\zeta_i}\big(y^{j+r-1} \: \sec_j(f)\cdot dy \big)$$
for all index $i$. Thanks to Proposition~\ref{prop:rescansecj},
we obtain:
$$\sres_{z_i, j,\can}(f) =
\rho_{\zeta_i}\big(P(y) \: y^{j+r-1} \: \sec_j(f)\cdot dy \big).$$
Now applying the residue formula with the function
$P(y) \: y^{j+r-1} \: \sec_j(f)$, we end up with:
$$\sum_{\substack{z \in \Fsep\\ z\neq 0}}
\sres_{z, j,\can}(f) =
-\rho_0\big(P(y) \: y^{j+r-1} \: \sec_j(f)\cdot dy \big)
-\rho_\infty\big(P(y) \: y^{j+r-1} \: \sec_j(f)\cdot dy \big).$$
The right hand side of the last formula can be computed explicity
on concrete examples (though determining a suitable polynomial $P(y)$
might be painful if the order of the poles are large). 
For example, when $\ord_{z,j}(f) \geq 0$, the first summand
$\rho_0\big(P(y) \: y^{j+r-1} \: \sec_j(f)\cdot dy \big)$ vanishes.

\subsection{Change of variables}

In this final subsection, we
analyse the effect of an endomorphism $\gamma$ of $\Frac(\calA)$ 
on the residues. According to Theorem~\ref{theo:endoOre}, $\gamma(X)
= CX$ for some $C \in \Frac(\calC)$ and we have:
$$\gamma\Big(\sum_i a_i X^i\Big) = \sum_i a_i \norm_i(C) X^i$$
where, by definition, $\norm_i(C) = C \cdot \theta(C) \cdots 
\theta^{i-1}(C)$. Define $Z = \gamma(Y)$. We have:
$$Z \,=\, \norm_r(C) \cdot Y 
\,=\, \norm_{\Frac(\calC)/\Frac(\calZ)}(C) \cdot Y 
\,\,\in\,\, \Frac(\calZ)$$
and $\gamma$ acts on $\Frac(\calC)$ through the change of
variables $Y \mapsto Z$.

\begin{defi}
Let $\gamma$ as above and let $z \in \Fsep$

\noindent
We say that $z$ is \emph{$\gamma$-regular} if $Z$ 
has no zero and no pole at $Y = z$.

\noindent
When $z$ is $\gamma$-regular, we define
$\gamma_\star z$ as the value taken by $Z$ at the point $Y = z$.
\end{defi}

For $f \in \Frac(\calC)$ and $z \in \Fsep$, we have the formula 
$$\res_{\gamma_\star z}\big(f{\cdot}dY\big) = 
\res_{z}\big(\gamma(f){\cdot}dZ\big) =
\res_{z}\left(\gamma(f)\:\frac{dZ}{dY}{\cdot}dY\right).$$
The aim of this subsection is to extend this relation to any $f \in 
\Frac(\calA)$, replacing classical commutative residues by skew 
residues.

\subsubsection{A general formula}

Comparing skew residues at $\gamma_\star z$ and $z$ is not 
straightforward 
because they do not live in the same space: the former lies in 
$\calA/N_1\calA$ where $N_1$ is the minimal polynomial of $\gamma_\star z$ 
while  the latter sits in $\calA/N_2\calA$ where $N_2$ is the minimal 
polynomial of $z$. We then first need to relate $\calA/N_1\calA$ and 
$\calA/N_2\calA$. For this, we remark that, as $\gamma$ acts through the 
change of variables $Y \mapsto Z$ on $\calZ$, it maps $N_1$ to a 
multiple of $N_2$. Therefore it induces a morphism of $K$-algebras 
$\calA/N_1\calA \to \calA/N_2\calA$.

\begin{theo}
\label{theo:chvar}
Let $\gamma : \Frac(\calA) \to \Frac(\calA)$ be an endomorphism
of $K$-algebras.
Let $z \in \Fsep$, $z \neq 0$ be a $\gamma$-regular point.

\begin{enumerate}[(i)]
\item For any admissible choice of $\taylisom_{\gamma_\star z}$
(see Definition~\ref{def:admissible})
there exists an admissible choice of $\taylisom_z$ such that:
\begin{equation}
\label{eq:chvar}
\gamma\circ \sres_{\gamma_\star z}(f) = 
\sres_z\left(\gamma(f) \:\frac{d\gamma(Y)}{dY}\right)
\end{equation}
for all $f \in \Frac(\calA)$.
\item A skew rational function $f \in \Frac(\calA)$ has a single
pole at $\gamma_\star z$ if and only if $\gamma(f)$ has a single
pole at $f$. When this occurs, Eq.~\eqref{eq:chvar} holds for any 
admissible choices of $\taylisom_{\gamma_\star z}$ and
$\taylisom_z$.
\end{enumerate}
\end{theo}

The following lemma will be used in the proof of
Theorem~\ref{theo:chvar}.

\begin{lem}
\label{lem:rescomposition}
Let $N \in \calZ$. Let $S \in (\calZ/N\calZ)[[T]]$ be a series
with constant term $0$. Let:
$$\begin{array}{rcl}
\psi: \quad
(\calA/N\calA)(\!(T)\!) & \longrightarrow & (\calA/N\calA)(\!(T)\!)
\smallskip \\
\sum_i a_i T^i & \mapsto & \sum_i a_i S^i.
\end{array}$$
For all $f \in (\calA/N\calA)(\!(T)\!)$, we have the formula:
\begin{equation}
\label{eq:rescomposition}
\res\left(\psi(f) \: \frac{dS}{dT}\right) = \res(f).
\end{equation}
\end{lem}

\begin{proof}
When $f \in (\calA/N\calA)[[T]]$, both sides of
Eq.~\eqref{eq:rescomposition} vanish and the conclusion of the lemma
holds. Moreover, since $\psi$ and $\res$ are both $K$-linear, it is
enough to establish the lemma when $f = T^i$ with $i < 0$.
Eq.~\eqref{eq:rescomposition} then reads
$\res\big(S^i \frac{dS}{dT}\big) = \res\big(T^i\big)$ and is a direct
consequence of the classical formula of change of variables for residues.
\end{proof}

\begin{proof}[Proof of Theorem~\ref{theo:chvar}]
We begin by some preliminaries. As before, we define $C = \gamma(X)\:
X^{-1}$ and $Z = \gamma(Y) = \norm_{\calC/\calZ}(C) \cdot Y$.
We put $z_1 = \gamma_\star z$ and $z_2 = z$. For $i \in \{1,2\}$,
we define $N_i$ as the minimal polynomial of $z_i$.
The quotient ring $\calZ/N_i\calZ$ is an algebraic separable extension 
of $F$; we will denote it by $E_i$ in the rest of the proof. 
By construction, $E_i$ admits a natural embedding into $\Fsep$ (obtained 
by mapping $Y$ to $z_i$). The fact that $\gamma$ acts on $\calZ$ by 
right composition by $Z$ shows that $\gamma_C$ induces a field inclusion 
$E_1 \hookrightarrow E_2$, which is compatible with the embeddings in 
$\Fsep$. In what follows, we shall always view $E_1$ and $E_2$ as 
subfields of $\Fsep$ with $E_1 \subset E_2$.

For $i \in \{1,2\}$, we recall that the Taylor expansion around $z_i$ 
provides us with a canonical isomorphism $\taylisom_i^\calZ : 
\hat\calZ_{N_i} \stackrel\sim\to E_i[[T]]$. The latter extends by
$K$-linearity to an isomorphism $\taylisom_i^\calC : 
\hat\calC_{N_i} \stackrel\sim\to K \otimes_F E_i[[T]]$.
We recall that $\taylisom_i^\calZ(Y) = \taylisom_i^\calC(Y) = z_i + T$. 
We set $S = \taylisom_2^\calZ(Z) - z_1$ and consider the mapping:
$$\begin{array}{rcl}
\varphi^\calZ: \quad
E_1[[T]] & \longrightarrow & E_2[[T]]
\smallskip \\
\sum_i a_i T^i & \mapsto & \sum_i a_i S^i.
\end{array}$$
We extend it by $K$-linearity to a map 
$\varphi^\calC: K \otimes_F E_1[[T]] \to K \otimes_F E_2[[T]]$.
We have:
$$\varphi^\calC \circ \taylisom_1^\calC(Y) = \varphi^\calC(z_1 + T) =
z_1 + S = \taylisom_2^\calC(Z) = \taylisom_2^\calC \circ \gamma(Y).$$
We deduce from this equality that the diagram
$$\xymatrix @C=4em {
\hat\calC_{N_1} \ar[r]^-{\taylisom_1^\calC}_-{\sim} 
\ar[d]_-{\gamma}
& K \otimes_F E_1[[T]] \ar[d]^-{\varphi^\calC} \\
\hat\calC_{N_2} \ar[r]^-{\taylisom_2^\calC}_-{\sim} 
& K \otimes_F E_2[[T]] }$$
is commutative, \emph{i.e.}
$\varphi^\calC \circ \taylisom_1^\calC = \taylisom_2^\calC \circ \gamma$.
Let us now consider a $z_1$-admissible choice of $\taylisom_{z_1}$ and
call it $\taylisom_1$ for simplicity. It is a prolongation of
$\taylisom_1^\calC$.
Besides, by Theorem~\ref{theo:endoquotOre}, there exists $C_1 \in 
(\calC/N_1\calC)[[T]] \simeq K \otimes_F E_1[[T]]$ such that 
$\taylisom_1 (X) = C_1 X$. The properties of $\taylisom_1$ 
ensure in addition that $C_1 \equiv 1 \pmod T$ and that:
$$\norm_{K \otimes_F E_1[[T]]/E_1[[T]]}\big(C_1\big) = 
\frac{\taylisom_1(Y)}{Y} = 1 + \frac T{z_1}$$
(see also Remark~\ref{rem:taylisomz}).
Applying $\varphi^\calC$ to this relation, we find:
\begin{equation}
\label{eq:norm1}
\norm_{K \otimes_F E_2[[T]]/E_2[[T]]}\big(\varphi^\calC(C_1)\big) = 
1 + \frac S{z_1} = \frac{\taylisom_2^\calZ(Z)}{z_1}.
\end{equation}
Let $\bar C \in \calC/N_2\calC \simeq K \otimes_F E_2$ be the reduction 
of $C$ modulo $N_2$. We shall often view $\bar C$ as a constant series 
in $(\calA/N_2\calA)[[T]]$. Since the norm of $C$ in the extension 
$\calC/\calZ$ is by definition $Z\:Y^{-1}$, we find:
\begin{equation}
\label{eq:norm2}
\norm_{K \otimes_F E_2[[T]]/E_2[[T]]}\big(\bar C\big) =
\norm_{K \otimes_F E_2/E_2}\big(\bar C\big) = \frac{z_1}{z_2}
\end{equation}
and:
\begin{equation}
\label{eq:norm3}
\norm_{K \otimes_F E_2[[T]]/E_2[[T]]}\big(\taylisom_2^\calC(C)\big) =
\taylisom_2^\calC\big(Z\:Y^{-1}\big) = 
\frac{\taylisom_2^\calZ(Z)}{z_2 + T}.
\end{equation}
Combining Eqs.~\eqref{eq:norm1}, \eqref{eq:norm2} and~\eqref{eq:norm3},
we obtain:
$$\norm_{K \otimes_F E_2[[T]]/E_2[[T]]}
\left(\frac{\bar C \cdot \varphi^\calC(C_1)}{\taylisom_2(C)}\right)
= 1 + \frac T{z_2}.$$
Set $C_2 = \frac{\bar C \cdot \varphi^\calC(C_1)}{\taylisom_2(C)}$ and
let $\taylisom_2 : \hat\calA_{N_2} \to (\calA/N_2\calA)[[T]]$ be
the morphism mapping $X$ to $C_2 X$. The above computations show that 
$\taylisom_2$ is well defined and coincide with $\taylisom_2^\calC$ on 
$\hat\calC_{N_2}$. On the other hand, one checks immediately that $C_2 
\equiv 1 \pmod{N_2}$, proving then that $\taylisom_2$ induces the 
identity modulo $N_2$. As a consequence, $\taylisom_2$ is an isomorphism and 
it is a $z$-admissible choice for $\taylisom_z$. Moreover, it sits in the
following commutative diagram:
$$\xymatrix @C=4em {
\hat\calA_{N_1} \ar[r]^-{\taylisom_1}_-{\sim} 
\ar[d]_-{\gamma}
& (\calA/N_1\calA)[[T]] \ar[d]^-{\varphi} \\
\hat\calA_{N_2} \ar[r]^-{\taylisom_2}_-{\sim} 
& (\calA/N_2\calA)[[T]] }$$
where $\varphi$ is the extension of $\varphi^\calC$ defined by
$\varphi\big(\sum_i a_i T^i\big) = \sum_i \gamma(a_i) S^i$. The
first assertion now follows from Lemma~\ref{lem:rescomposition} 
together with the fact that $\frac{dS}{dT} = \taylisom_2^\calZ
\big(\frac{dZ}{dY}\big)$.

The equivalence in assertion (ii) follows from what we have done before 
after noticing that $S$ has $T$-valuation $1$ by the regularity assumption 
on $z$. The fact that Eq.~\eqref{eq:chvar} holds for any $\gamma_star z$-admissible
choices of $\taylisom_{\gamma_\star z}$ and $\taylisom_z$ in this case
is a direct consequence of the fact that skew residues do not depend on
the choice of the Taylor isomorphisms when poles are simple.
\end{proof}

\subsubsection{The case of canonical residues}

We recall that, when $p$ does not divide $r$, there is a distinguished 
choice for $\taylisom_z$ leading to a notion of canonical skew residues, 
denoted by $\sres_{z,\can}$. 
After Theorem~\ref{theo:chvar}, one could hope that Eq.~\eqref{eq:chvar}
always holds with canonical residues, as the latter are canonical.
Unfortunately, it is not that simple in general. However, there is
an important case where our first naive hope is correct.

\begin{theo}
\label{theo:chvarcan}
We assume that $p$ does not divide $r$.

\noindent
Let $\gamma : \Frac(\calA) \to \Frac(\calA)$ be an endomorphism
of $K$-algebras.
Let $z \in \Fsep$, $z \neq 0$ be a $\gamma$-regular point.
\emph{If $\gamma(X) \: X^{-1} \in \Frac(\calZ)$}, we have:
$$\gamma\circ \sres_{\gamma_\star z,\can}(f) = 
\sres_{z,\can}\left(\gamma(f) \:\frac{d\gamma(Y)}{dY}\right)$$
for all $f \in \Frac(\calA)$.
\end{theo}

\begin{proof}
After Theorem~\ref{theo:chvar}, it is enough to check that the admissible 
choice $\taylisom_{\gamma_\star z,\can}$ leads to the admissible choice
$\taylisom_{z,\can}$. By Theorem~\ref{theo:taylorcan}, this reduces 
further to check that $C_2$ lies in $(\calZ/N_2\calZ)[[T]]$ as soon as 
$C_1$ is in $(\calZ/N_1\calZ)[[T]]$ (with the notations of the proof of 
Theorem~\ref{theo:chvar}). This is obvious from the definition of $C_2$.
\end{proof}

We now consider the general case. 
Proposition~\ref{prop:taylorunique} tells us that different choices of 
$\taylisom_z$ are conjugated. As a consequence, $\sres_{\gamma_\star 
z}(f)$ and $\sres_{z,\can}\big(\gamma(f) \frac{d \gamma(Y)}{dY}\big)$ 
should be eventually related up to some conjugacy. In the present 
situation, it turns out that the conjugating function can be explicited.
From now on, we assume that $p$ does not divide $r$.
As before, we consider an endomorphism of $K$-algebras $\gamma :
\Frac(\calA) \to \Frac(\calA)$ and we define $C = \gamma(X)\:X^{-1}
\in \Frac(\calC)$. 
We introduce the extension $\calZ'$ of $\Frac(\calZ)$ obtained by
adding a $r$-th root of $\norm_{\Frac(\calC)/\Frac(\calZ)}(C)$ and
form the tensor products $\calC' = \calZ' \otimes_{\calZ} \calC$
and $\calA' = \calZ' \otimes_{\calZ} \calA$. 
We emphasize that $\calC'$ is not a field in general but a product
of fields. However, the extension $\calC'/\calZ'$ is a cyclic 
Galois covering of degree $r$ whose Galois group is generated by the 
automorphism $\id \otimes \theta$.
Similarly, $\calA'$ could be not isomorphic to an algebra of skew 
rational functions. Nevertheless, we have the following lemma.

\begin{lem}
\label{lem:extensiontaylisom}
Given a $\gamma$-regular point $z \in \Fsep$ and its minimal polynomial $N 
\in \calZ^+$, any admissible isomorphism
$\taylisom_z : \hat\calA_N \stackrel\sim\longrightarrow
(\calA/N\calA)[[T]]$
extends uniquely to an isomorphism:
$$\taylisom_z^{\calA'} : 
\calZ' \otimes_\calZ \hat\calA_N \stackrel\sim\longrightarrow 
 (\calA'/N\calA')(\!(T)\!)$$
inducing the identity after reduction modulo $N$ on the left
and modulo $T$ on the right.
\end{lem}

\begin{proof}
Let us first prove an analogous statement for
$\taylisom_z^\calZ : \hat\calZ_N \to (\calZ/N\calZ)[[T]]$.
For simplicity, set 
$Z_0 = \norm_{\Frac(\calC)/\Frac(\calZ)}(C) \in \calZ$
and let $\bar Z_0$ be the reduction of $Z_0$ modulo $N$. By the 
regularity assumption, $\bar Z_0 \neq 0$. 
Hence $\taylisom_z^\calZ(Z_0)$ has a unique $r$-th root in 
$(\calZ'/N\calZ')[[T]]$ whose constant term is the image of 
$\sqrt[r]{Z_0}$ in $\calZ'/N\calZ'$. This basically proves the
existence and the unicity of a prolongation $\taylisom_z^{\calZ'}$
of $\taylisom_z^\calZ$.

Now, a prolongation of $\taylisom_z$ is given by $\taylisom_z^{\calA'} = 
\taylisom_z^{\calZ'} \otimes \taylisom_z$, which proves the existence. For 
unicity, we remark that, by unicity of $\taylisom_z^{\calZ'}$, any 
isomorphism $\taylisom_z^{\calA'}$ satisfying the conditions of the 
lemma has to coincide with $\taylisom_z^{\calZ'}$ on
$\calZ' \otimes_{\calZ} \hat\calZ_N$. Since $\taylisom_z^{\calA'}$
is a ring homomorphism, we deduce that it necessarily agrees with
$\taylisom_z^{\calZ'} \otimes \taylisom_z$ on its domain of
definition. Unicity follows.
\end{proof}

\noindent
Lemma~\ref{lem:extensiontaylisom} shows that the function
$\sres_{z,\can} : \Frac(\calA) \to \calA/N\calA$ admits a canonical
extension to $\calC'$. We will continue to call it $\sres_{z,\can}$
in the sequel.
We now consider the element:
$$C' = 
\frac C {\sqrt[r]{\norm_{\Frac(\calC)/\Frac(\calZ)}\big(C\big)}}
\,\in\, \calC'.$$
By construction, it has norm $1$ in the extension $\calC'/\calZ'$.
Hilbert's Theorem 90 then guarantees the existence of an invertible
element $U \in \calC'$ (uniquely determined up to multiplication by
an element of $\calZ'$) such that:
\begin{equation}
\label{eq:Cprime}
C' = \frac{(\id \otimes \theta)(U)} U.
\end{equation}

\begin{rem}
Raising Eq.~\eqref{eq:Cprime} to the $r$-th power, we get:
$$\frac{(\id \otimes \theta)(U^r)}{U^r} =
(C')^r = \frac{(\id \otimes \theta)(V)} V
\quad \text{with} \quad
V = \prod_{i=0}^{r-1} \: \theta^i\big(C\big)^{i+1-r}.$$
Therefore $U^r \in V \calZ'$. This observation gives an alternative option
for finding $U$: we look for an element $Z' \in \calZ'$ for which $V Z'$ 
is the $r$-th power in $\calC'$ and we extract its $r$-th root.
\end{rem}

\begin{theo}
\label{theo:chvarcan2}
With the above notations, we have:
$$\gamma\circ \sres_{\gamma_\star z,\can}(f) = 
U^{-1} \cdot \sres_{z,\can}\left(U \: \gamma(f) \: U^{-1}
\:\frac{d\gamma(Y)}{dY}\right) \cdot U$$
for all $\gamma$-regular point $z \in \Fsep$, $z \neq 0$ and all $f \in 
\Frac(\calA)$.
\end{theo}

\begin{rems}
\begin{enumerate}[(1)]
\item When $C \in \Frac(\calZ)$, the norm of $C$ is equal to $1$, so 
that we have $\calC' = \Frac(\calC)$ and $C' = 1$. In this case, one can 
take $U = 1$ and the statement of Theorem~\ref{theo:chvarcan2} reduces to 
that of Theorem~\ref{theo:chvarcan}.

\item When $f \in \Frac(\calC)$, $\gamma(f)$ also lies in $\Frac(\calC)$ 
and thus commutes with $f$. Hence, the product $U \: \gamma(f) \: U^{-1}$ 
reduces to $\gamma(f)$. Similarly the skew residue 
$\sres_{z,\can}\big(\gamma(f) \:\frac{d\gamma(Y)}{dY}\big)$ is an element 
of $\calC/N_2\calC$ and thus also commutes with $U$. Finally, 
Theorem~\ref{theo:chvarcan2} reads in this case:
$$\gamma\circ \sres_{\gamma_\star z,\can}(f) = 
\sres_{z,\can}\left(\gamma(f) \:\frac{d\gamma(Y)}{dY}\right)$$
which is the usual formula for commutative residues.
\end{enumerate}
\end{rems}

\begin{proof}[Proof of Theorem~\ref{theo:chvarcan2}]
We keep the notations of the proof of Theorem~\ref{theo:chvarcan} and
assume in addition that the isomorphism $\taylisom_{\gamma_\star z}$
we started with is $\taylisom_{\gamma_\star z, \can}$, \emph{i.e.}:
$$C_1 = \left( 1 + \frac T{z_1}\right)^{\!1/r}.$$
By the proof of Theorem~\ref{theo:chvar}, Eq.~\eqref{eq:chvar} holds
when $\taylisom_z$ is defined by $\taylisom_z(X) = C_2 X$ with:
$$C_2 = \frac{\bar C}{\taylisom_z^\calC(C)} \cdot 
\left( 1 + \frac S{z_2}\right)^{\!1/r}.$$
Here we recall that $\bar C$ is the image of $C$ in $\calC/N_2
\calC$ and $S = \taylisom_2(Z) - z_2$ where $Z$ was defined by
$Z = \norm_{\Frac(\calC)/\Frac(\calZ)}(C) \cdot Y$.
On the other hand, the isomorphism $\taylisom_{z,\can}$ is defined
by:
$$\taylisom_{z,\can}(X) = \left( 1 + \frac T{z_2}\right)^{\!1/r} X.$$
Let $\bar C'$ and $\bar U$ be the image of $C'$ and $U$ in 
$\calC'/N_2\calC'$ respectively.
We consider the ring homomorphism $\taylisom : \calZ' \otimes_\calZ 
\hat\calA'_N \to (\calA'/N\calA')[[T]]$ defined by:
\begin{equation}
\label{eq:deftau}
\taylisom(f) =  
\bar U^{-1} \cdot \taylisom_{z,\can}^{\calA'}\big(U g \: U^{-1}\big)
\cdot \bar U
\end{equation}
for $g \in \hat\calA_N$.
A simple computation shows that $\taylisom(X) = QX$ with:
\begin{align*}
Q 
&= \frac{\id \otimes\theta(\bar U)}{\bar U} \cdot 
\taylisom_{z,\can}^{\calA'}\left(\frac{U}{\id\otimes\theta(U)}\right) \cdot 
\left( 1 + \frac T{z_2}\right)^{\!1/r}\\
&= \bar C' \cdot 
\taylisom_{z,\can}^{\calA'}\left(
\frac{\sqrt[r]{\norm_{\Frac(\calC)/\Frac(\calZ)}(C)}}{C}\right) \cdot 
\left( 1 + \frac T{z_2}\right)^{\!1/r}.
\end{align*}
Raising this equality to the $r$-th power, we get:
\begin{align*}
Q^r 
& = (\bar C')^r \cdot 
\taylisom_z^\calC\left(
\frac{\norm_{\Frac(\calC)/\Frac(\calZ)}(C)}{C^r} \right)\cdot
\left(1 + \frac T{z_2}\right) \smallskip \\
& = (\bar C')^r \cdot 
\taylisom_z^\calC\left(
\frac{Z}{Y} \: \frac 1{C^r} \right)\cdot
\left(1 + \frac T{z_2}\right).
\end{align*}
Noticing that $\taylisom_z^\calC(Y) = z_2 + T$ and
$\taylisom_z^\calC(Z) = z_1 + S$, we obtain:
\begin{equation}
\label{eq:Qr}
Q^r = \frac{z_1}{z_2} \cdot
\left(\frac{\bar C'}{\taylisom_z^\calC(C)} \right)^{\!r}\cdot 
\left(1 + \frac S{z_1}\right).
\end{equation}
Now, observe that the identity $(C')^r = C^r \:\frac Y Z$ gives 
$(\bar C')^r = \bar C^r \: \frac {z_2}{z_1}$ after reduction modulo $N_2$.
Putting this input in Eq.~\eqref{eq:Qr}, we finally find:
$$Q^r = 
\left(\frac{\bar C}{\taylisom_z^\calC(C)} \right)^{\!r}\cdot 
\left(1 + \frac S{z_1}\right) = C_2^r$$
Besides, a direct computation shows that both series $Q$ and $C_2$
have a constant coefficient equal to~$1$. Therefore, the equality
$Q^r = C_2^r$ we have just proved implies $Q = C_2$. In other words
$\taylisom(X) = \taylisom_z(X)$. Since moreover $\taylisom$ and
$\taylisom_z$ agree on $\sqrt[r]{\norm_{\Frac(\calC)/\Frac(\calZ)}(C)}$,
they coincide everywhere.
Coming back to the defintion of $\taylisom$ (see 
Eq.~\eqref{eq:deftau}), we obtain:
$$\sres_z\big(g) 
= \bar U^{-1} \cdot \sres_{z,\can}\big(U g \: U^{-1}\big) \cdot \bar U
= U^{-1} \cdot \sres_{z,\can}\big(U g \: U^{-1}\big) \cdot U$$
for all $g \in \Frac(\calA)$.
Specializing this equality to $g = \gamma(f) \: \frac{d \gamma(Y)}{dY}$, 
we get the theorem.
\end{proof}

\end{document}